\documentclass[]{interact}                                                       

\usepackage[caption=false]{subfig}

\usepackage[natbibapa,nodoi]{apacite}
\usepackage{xcolor}

\theoremstyle{plain}
\newtheorem{theorem}{Theorem}[section]

\newtheorem{corollary}[theorem]{Corollary}

\theoremstyle{definition}
\newtheorem{definition}[theorem]{Definition}

\theoremstyle{remark}

\title{\LARGE \bf
State estimator design using Jordan based long short-term memory networks
}

\author{
\name{Avneet Kaur\textsuperscript{a} and Kirsten Morris\textsuperscript{a}\thanks{CONTACT Kirsten Morris. Email: kmorris@uwaterloo.ca}}
\affil{\textsuperscript{a}Department of Applied Mathematics, University of Waterloo, Ontario, Canada}
\thanks{*This work was supported by NSERC and the Faculty of Mathematics, University of Waterloo}%
}

\begin{document}
\maketitle
\thispagestyle{empty}
\pagestyle{empty}

\begin{abstract}
State estimation of a dynamical system refers to estimating the state of a system given an imperfect model, noisy measurements and some or no information about the initial state. While Kalman filtering is optimal for estimation of linear systems with Gaussian noises, calculation of optimal estimators for nonlinear systems is challenging. We focus on establishing a pathway to optimal estimation of high-order systems by using recurrent connections motivated by Jordan recurrent neural networks(JRNs). The results are compared to the corresponding Elman structure based long short-term memory network(ELSTM) and the KF for linear and EKF for nonlinear systems. The results suggest that for nonlinear systems, the use of long short-term memory networks can improve estimation error and also computation time. Also, the Jordan based long short-term memory networks(JLSTMs) require less training to achieve performance  similar to ELSTMs.
\end{abstract}

\begin{keywords}
    state estimation, machine learning, neural networks, long short-term memory networks
\end{keywords}

\section{Introduction}
The problem of estimating the state of a dynamical system given a noisy model, noisy measurements and partial information about the initial state is called state estimation. In practical applications estimating the state is a challenging task for various reasons. The complexity of high-order systems, availability of a finite number of measurements, inaccurate measurements, incomplete information about the initial state and errors in model can make it difficult to estimate the state of the system. Some review papers on  state estimation are \cite{kutschireiter2020hitchhiker, bernard2022observer, jin2021new, trends_of_estimation}.

Various model-based approaches to the problem of state estimation exist. The Kalman filter(KF) in  \cite{kalman1960new} is a recursive and easy to implement filter. It is used for optimal state estimation of linear systems with Gaussian stochastic noises. The most obvious generalization of this filter to nonlinear systems is the extended Kalman filter(EKF). The EKF relies on linearization of the system at each time-step. Such a generalization does not work well for many nonlinear systems and only local convergence guarantees are available (see for example, \cite{krener_2002,AfsharGermMorris}). Furthermore, the EKF is not  optimal. Another common approach is to use the unscented Kalman filter which has lesser linearization error than the EKF. The idea is to chose some specific sample points instead of using the entire density function to propagate means and co-variances. The nonlinear transformation is then applied only on the sample points leading to a decrease in the linearization error. But it is still not an optimal method and can be computationally intensive. To perform state estimation in an optimal manner, the Kalman filter approach is extended in  \cite{mortensen} to continuous-time nonlinear state-space systems. The idea is to estimate the initial condition noise and process noise which in turn provides the optimal state estimator. Certain assumptions ensure that an optimum for the maximum likelihood cost function for continuous-time systems exists.  \cite{moireau} extended  the results in \cite{mortensen} to discrete-time state space systems. However, this approach involves solving the Hamilton-Jacobi Bellman equation which in itself is a challenging task as it is a second-order nonlinear partial differential equation whose solution is  affected by the curse of dimensionality (see for example, \cite{krener_2003, pequito,kang2017mitigating, nakamura2021adaptive}). Thus, even though optimal, this approach is not widely used. 

Using only input-output data to identify states is called system identification or trajectory tracking. A major difference between these methods and state estimation is that they do not use the model, even when it is available. Recurrent neural networks have been shown to work well for system identification (see for example, \cite{park2020analysis}. Jordan recurrent neural networks(JRNs) have been used for system identification by several researchers, including \cite{kasiran2012mobile, wu2019time},  and have been shown to work as well as Elman recurrent neural networks(ERNs). Convergence analysis, with some assumptions, of JRNs as well as ERNs for system identification has been discussed in \cite{kuan_convergence_1994}. 

To predict future values based on observed trends in data is called time-series analysis or forecasting. Recurrent neural networks have been used for forecasting from several years. ERNs provided the best estimation properties when compared with the method of least squares and feedforward neural networks for noisy time series data with fully observable states in \cite{gencay1997nonlinear}. In \cite{llerena2021forecasting}, an Elman-based long short-term memory (ELSTM) network architecture is explored for different systems and compared with an EKF for filtering. The authors conclude that the proposed ELSTM structure performs better than the EKF for nonlinear systems. In \cite{kandiran2019comparison}, authors show that the ERN performs at least as well as a feedforward neural network for Lyapunov exponents forecasting, discussing cases where it performs better in detail as well. 

 Since a model, when available, can help improve  estimate accuracy,  methods combining model as well as data are attractive. Machine learning approaches to state estimation are being explored. A machine learning framework that overcame the requirement to find solutions to the Hamilton-Jacobi Bellman partial differential equation in the continuous-time case is discussed in \cite{kunisch}. A neural network to approximate the gradient of the solution of Hamilton-Jacobi Bellman equation is used. The equation is exploited to construct the filter which is referred to as the observer gain. It is shown that the results coincide with the Kalman-Bucy filter for a linear problem. A hybrid approach, which utilised the KF framework of correction and prediction for discrete-time systems, is discussed in \cite{gao_2019}.  Even though the focus of this work was trajectory tracking, the modified assumptions and use of model, led to the problem being that of state estimation. Recurrent neural networks are used, in particular, long short-term memory networks, to approximate the state from observations, initial data and noise estimates. A two step process of prediction and correction is used and two frameworks are presented, one Bayesian and the other non-Bayesian. For the Bayesian approach, the neural network is designed so that it predicts and filters iteratively within the same network. This framework is supported by using Bayes rule. For the non-Bayesian approach two deep ELSTM networks are used, one network for prediction and the other for filtering, based on the fact that neural networks are individually capable of performing both the steps independently. Thus, one network performed prediction without measurements and the other network corrected the predictions by using new measurements. This approach is considered hybrid as the model is used for training data generation. \cite{ramos} used nonlinear regression on data generated by solving the system and estimator dynamics to approximate a Luenberger estimator. A Luenberger estimator for discrete-time nonlinear systems  is found by approximating the mapping using unsupervised learning in \cite{dl_based_observer}. In \cite{benosman_borggaard_red_order}, a method to design robust, low-order observers for a class of spectral infinite-dimensional nonlinear systems was discussed. This method included a data-driven learning algorithm to optimise observer performance. Jun Fu \cite{zhijunfu} proposed a multi-time scales neural-network based adaptive  estimator design. A deep-learning method to solve the Hamilton-Jacobi Bellman equation for state estimation was proposed in \cite{adhyaru}. A neural network was used to approximate the solution, which further helped find the estimator gain. \cite{benosman_borggaard_red_order} discuss a robust, low order estimator using principal orthogonal decomposition(POD)  for a class of spectral infinite-dimensional nonlinear systems. A data-driven method  is used to auto-tune the estimator gains. The results are tested on 1D Burgers equation \cite{benosman_borggaard_red_order} and 2D Boussinesq equations \cite{benosman_borggaard_red_order_2}.

Reinforcement learning-based approaches are also becoming popular due to recent advancements in the field. In \cite{rl_paper}, the authors proposed a reinforcement learning framework combined with a KF for tracking ground vehicles using an integrated navigation system, and called it an adaptive KF navigation algorithm. The state of the system  was estimated using a KF and then used  to calculate rewards for a reinforcement learning framework which updated the process noise covariance. In \cite{patrick} the problem is viewed as a Bayes-adaptive Markov decision process and solved online using Monte Carlo tree search with an unscented KF to account for process noise and parameter uncertainty. For details about Monte Carlo tree search see \cite{rl_book}. 

Use of a reduced-order system is a common approach since it helps overcome the curse of dimensionality. A deep state estimator which used a neural network methodology to develop a nonlinear relationship between the measurements and a reduced order state was developed in \cite{nair_goza}. This relationship was commonly approximated as linear and then they used the advancements in the field of deep learning to make better approximations. The approximated reduced-order state was used as an initial condition to estimate the full state. His proposed methodology outperformed common linear estimation algorithms when tested on 1D Burgers equation and 2D Boussinesq equations. Several neural networks for state estimation of low-order systems were proposed in \cite{xie_deep_2021}. A recurrent neural network based on the principles of a decoupled EKF for state estimation was introduced in \cite{yadaiah} and used on reduced order models.

While feedforward neural networks do not link outputs or the hidden states at the previous time-step to the one at the next time-step, recurrent neural networks have varying structures linking outputs and/or hidden states at different time-steps. This property is of great importance for us, since in dynamical systems,  the state at the previous time-step, plus any inputs, determine the state at the next time-step. Another advantage of recurrent neural networks is that they can have a nonlinear activation function which is helpful for estimating the states of nonlinear systems. While this suggests that JRNs would be a good choice for state estimation of sequences, when it comes to longer sequences, training them can be extremely time-consuming (see for example, \cite{manaswi2018rnn}). This motivates the use of long short-term memory networks. Furthermore, choosing the right type of connections in a network is important. This can help the network to converge faster leading to lesser training time. It can also help in stability analysis and convergence analysis. 

In this work, we combine the idea of a state estimation algorithm with the tools of machine learning. Section \ref{sec:rnns} discusses the structure of an ERN and a JRN followed by a universal approximation theorem for state estimation using JRNs. The  Jordan recurrent network is more closely related to the structure of a dynamical system than Elman networks which suggests that it might be more appropriate for state estimation. The Elman  long short-term memory network (ELSTM) framework is thus extended to a Jordan-based long short-term memory (JLSTM) network  in hope that this structure will have reduced training time.  In Section \ref{sec:implementation} implementation of  ELSTM and JLSTM is described. Implementation of both structures was done similarly to ensure valid comparison.  In section \ref{sec:numerical_examples}, both structures are used to construct estimators for a number of examples.  In Section \ref{sec:results}, the  results of  ELSTM, JLSTM and KF/EKF are discussed. Both neural network-based estimators compare well to the traditional Kalman filters.
However, the JLSTM network takes lesser training time as compared to an ELSTM to achieve similar performance. 

\section{State estimation using simple recurrent neural networks}
\label{sec:rnns}
This work considers noisy discrete-time state space systems
\begin{align}
    \begin{split}
        x^{(t+1)}&=f(x^{(t)})+\omega^{(t+1)}  \\
        y^{(t+1)}&=h(x^{(t+1)})+\nu^{(t+1)} \\
        x^{(0)}&=x^0+\bar{x}^0 
        \label{eqn:noisy_disc_state_space_sys}
    \end{split}
\end{align}
where $x^{(t+1)}\in \mathcal{X} \subseteq \mathbb{R}^n, n \in \mathbb{N}$ is the state vector of the system at time-step $t+1, y^{(t+1)}\in \mathcal{Y} \subseteq \mathbb{R}^m, m \in \mathbb{N}$ is the measurement vector at time-step $t+1, \omega^{(t+1)}$ is the process noise vector at time-step $t+1$, $\nu^{(t+1)}$ is the measurement noise vector at time-step $t+1$, $x^0$ is the true initial condition and $\bar{x}^0$ is the initial condition noise. The functions $f$ and $h$ are known and  continuously differentiable. The aim of state estimation is to estimate the state $x^{(t+1)}$ based on new measurement $y^{(t+1)},$ previous state estimate $\hat{x}^{(t)}$, and  some information on the distribution of the  initial condition $x^{(0)}.$ 

\begin{definition}[Activation function] 
    A function $\sigma(\cdot): \mathbb{R} \rightarrow [a,b]$ is called an activation function, if $\sigma$ is monotonically increasing, $\lim\limits_{z \to -\infty} \sigma(z)=a$ and $\lim\limits_{z \to \infty} \sigma(z)=b.$
\end{definition}

\begin{figure*}
\centering
    \framebox{\includegraphics[scale=0.09]{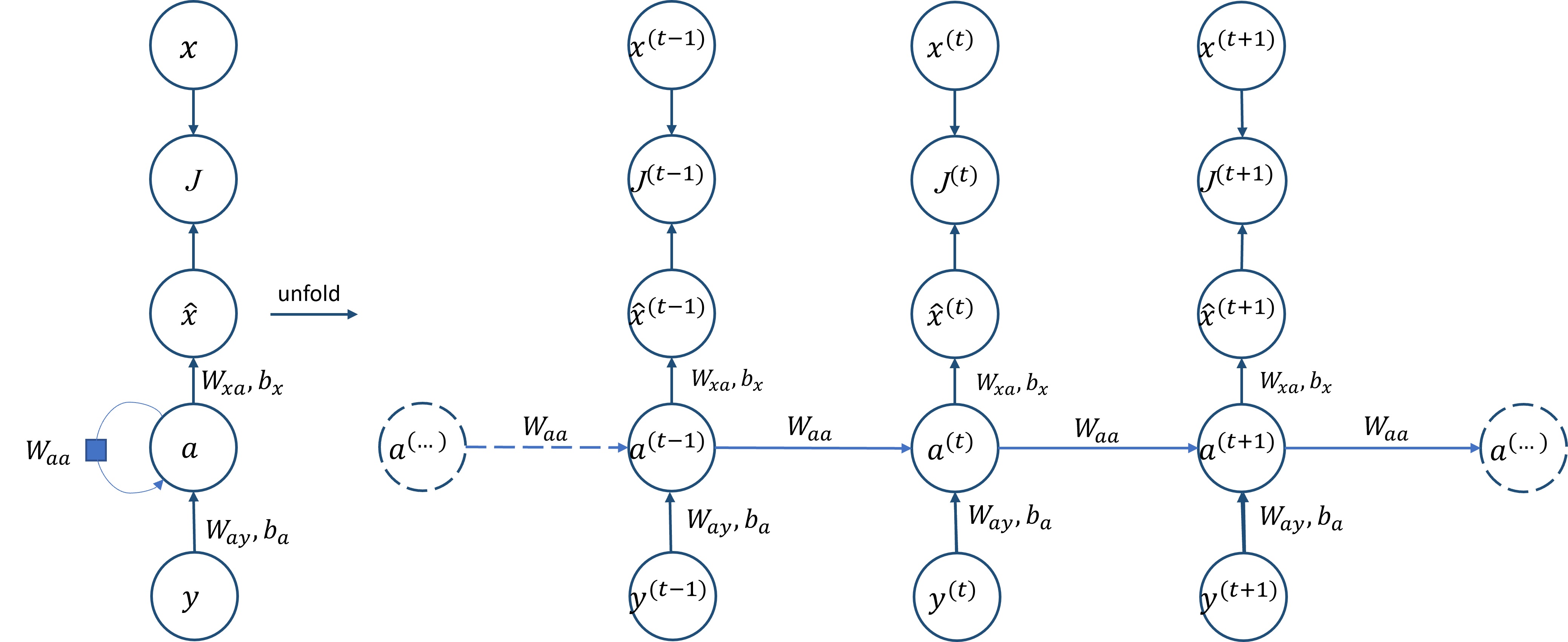}}
    \caption{The structure of an Elman recurrent neural network(ERN) for state estimation. It uses hidden to hidden recurrent connections. The symbols $y^{(t)}, a^{(t)}, x^{(t)}$ and $\hat{x}^{(t)}$ represent input measurement vector, hidden layer vector, true state vector and estimated state vector respectively at time $t$. The cost function $J$ is considered to be mean squared error(MSE). The weights and biases are represented by $W_{ay}, W_{aa}, W_{xa}, b_a$ and $b_x.$}
    \label{fig:ern}
\end{figure*}
 An Elman recurrent network (ERN) can been used for state estimation. In an ERN, the hidden layer at the previous time-step connects to the hidden layer at the next time-step(Fig. \ref{fig:ern}). ERNs are proven to be universal approximators for dynamical systems in \cite{schafer2007recurrent}. For state estimation, one can choose the inputs of the ERN at time-step $t$ to be the measurements at time-step $t$ and the outputs of the ERN to be the estimated state at time-step $t.$ Thus, an ERN  acts like a filter. The forward propagation of an ERN for state estimation is 
\begin{align}
    \begin{split}
        a^{(t)}&=\sigma(W_{ay} y^{(t)}+W_{aa} a^{(t-1)}+b_a) \\
        \hat{x}^{(t)}&=W_{xa} a^{(t)}+b_x
        \label{eqn:ern}
    \end{split}
\end{align}
where $\sigma(\cdot)$ is the activation function and $\hat{x}^{(t)}$ is the estimated state vector at time-step $t$. The hidden layer vector at time-step $t$ is $a^{(t)}.$ The measurement vector at time-step $t$ is denoted by $y^{(t)}.$ The weights and biases of the network are represented by $W_{ay}, W_{aa}, W_{xa}, b_a$ and $b_x$ respectively.
The approximation theorem for ERNs of a dynamical system presented in \cite{schafer2007recurrent} can be extended to estimation with slight modifications.  Since the proof is very similar to that in \cite{schafer2007recurrent}, it is omitted. 

In a JRN,  instead of using the previous hidden vector in calculation of the updated estimate,  the  estimate at the previous  time-step $t$ is to update the state  (Fig. \ref{fig:jrn}).  Thus, a Jordan-based recurrent neural network structure mimics the structure of a dynamical system \eqref{eqn:noisy_disc_state_space_sys}.   This suggests that training and accuracy of an estimator  may be improved over other structures by using a JRN. 
The forward propagation of a JRN for state estimation(Fig.\ref{fig:jrn}) is 
\begin{align}
    \begin{split}
        a^{(t)}&=\sigma(W_{ay} y^{(t)}+W_{ax} \hat{x}^{(t-1)}+b_a) \\
        \hat{x}^{(t)}&=W_{xa} a^{(t)}+b_x
        \label{eqn:jrn}
    \end{split}
\end{align}
where $\sigma(\cdot)$ is the activation function and $\hat{x}^{(t)}$ is the estimated state vector at time-step $t$. The hidden layer vector at time-step $t$ is given by $a^{(t)}.$ The weights and biases of the network are represented by $W_{ay}, W_{ax}, W_{xa}, b_a$ and $b_x$ respectively.

\begin{figure*}
    \centering
    \framebox{\includegraphics[scale=0.09]{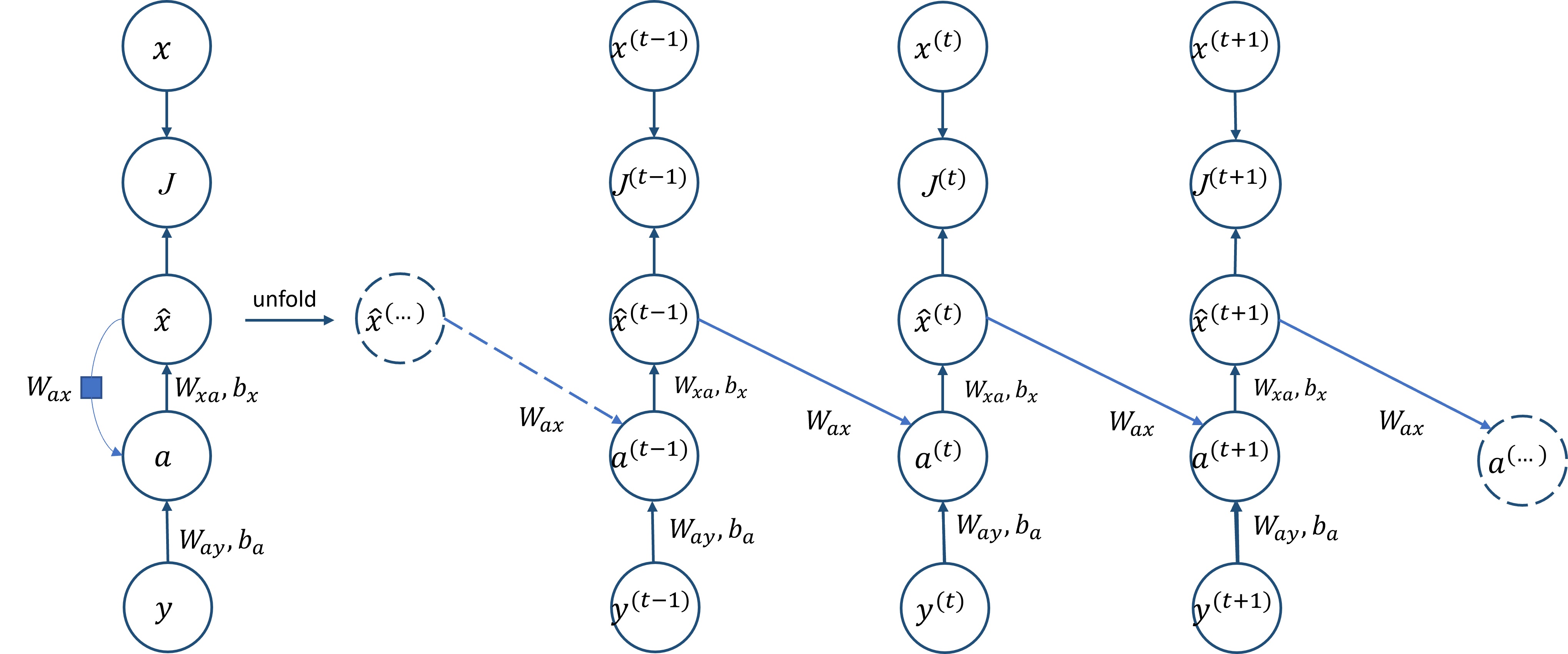}}
    \caption{The structure of the Jordan recurrent network(JRN) proposed for state estimation. It uses output-to-hidden recurrent connections similar to the filter's dynamical system. The symbols $y^{(t)}, a^{(t)}, x^{(t)}$ and $\hat{x}^{(t)}$ represent input measurement vector, hidden layer vector, true state vector and estimated state vector respectively at time $t$. Cost function $J^{(t)}$ is the mean squared error at time $t$. The weights and biases of the network are represented by $W_{ay}, W_{ax}, W_{xa}, b_x$ and $b_y.$ }
    \label{fig:jrn}
\end{figure*}

It will now be shown that JRNs are universal approximators for an estimator. 
We first recall the definitions and main results first presented in \cite{hornik1989multilayer} and later mentioned in \cite{schafer2007recurrent}. 

\begin{definition}
    Let $A^{n},n \in \mathbb{N}$ be the set of all affine functions  $A(\cdot):\mathbb{R}^{n} \rightarrow \mathbb{R}$ defined as $A(z)=w.z-b$ where $w,z \in \mathbb{R}^{n}, b \in \mathbb{R}$ and ‘.’  denotes dot product.
\end{definition}

\begin{definition}
    For any Borel-measurable function $\sigma(\cdot): \mathbb{R} \rightarrow \mathbb{R}$ and $n \in \mathbb{N}$ define 
    \begin{multline}
    \Sigma^{n}(\sigma)= \{(NN: \mathbb{R}^{N} \rightarrow \mathbb{R}): NN(z)=\sum_{j=1}^{J}\theta_j \sigma(A_{j}(z)), \\
    z\in \mathbb{R}^{n}, \theta_j \in \mathbb{R}, A_{j} \in A^{n}, J \in \mathbb{N}\}.
    \end{multline}
\end{definition}
The function class $\Sigma^{n}(\sigma)$ can be written in matrix form as 
$$NN(z)=\theta \sigma(Wz-b),$$
where $x \in \mathbb{R}^{n}, \theta,b \in \mathbb{R}^{\tilde{h}}$ and $W \in \mathbb{R}^{\tilde{h} \times n}.$ Here function $\sigma$ is applied component-wise. Similarly, the for function class $JRN^{m,n}(\sigma), \sigma$ is defined component-wise.

    Let $\mathcal{M}^n$ and $\mathcal{C}^n$ be the set of all Borel-measurable and continuous functions (both from $\mathbb{R}^n \rightarrow \mathbb{R}$) respectively. Let the Borel $\sigma$ algebra of $\mathbb{R}^N$ be denoted by $\mathbb{B}^n.$ Thus, $(\mathbb{R}^n, \mathbb{B}^n)$ refers to the $n$-dimensional Borel measurable space.

Clearly, for every continuous Borel-measurable function $\sigma$, $$\Sigma^n(\sigma) \subset \mathcal{C}^n \subset \mathcal{M}^n.$$ 

\begin{definition}
    Let $(X,\rho)$ be a metric space and $U,T \subseteq X.$ Then $U$ is $\rho-$dense in $T,$ if $\forall \epsilon >0$ and $\forall t \in T, \exists u \in U,$ such that $\rho(u,t)<\epsilon.$
\end{definition}

\begin{definition}
    A subset $U$ of $\mathcal{C}^n$ is said to be uniformly dense on a compact domain in $\mathcal{C}^n$ if for every compact subset $K \subseteq \mathbb{R}^n, U$  is $\rho_K$-dense in $\mathcal{C}^n,$ where for $g_1,g_2 \in \mathcal{C}^n,$ 
    \begin{equation*}
        \rho_K(g_1,g_2)\equiv \sup_{x \in K}|g_1(x)-g_2(x)|.
    \end{equation*}
\end{definition}

\begin{definition}
    Let $\mu$ be a probability measure on $(\mathbb{R}^n, \mathbb{B}^n).$ Functions $g_1,g_2 \in \mathcal{M}^n$ are said to be $\mu$-equivalent if $\mu\{x\in\mathbb{R}^n:g_1(x)=g_2(x)\}=1.$ 
\end{definition}

\begin{definition}
    Define the metric $\rho_{\mu}: \mathcal{M}^n \times \mathcal{M}^n \rightarrow \mathbb{R}^{+}$ on $(\mathbb{R}^n,\mathbb{B}^n)$ as  
    \begin{equation}
        \rho_{\mu}(g_1,g_2)=\inf\{\epsilon>0:\mu\{x\in \mathbb{R}^n:|g_1(x)-g_2(x)|>\epsilon\}<\epsilon\}
    \end{equation} 
    where $\mu$ is a probability measure on $(\mathbb{R}^n,\mathbb{B}^n).$
\end{definition}
 Note that if functions $g_1$ and $g_2$ are $\mu$-equivalent, then $\rho_{\mu}(g_1,g_2)=0.$

\begin{theorem}\cite{hornik1989multilayer} (Universal approximation theorem for feed-forward neural networks)
\label{theorem:ffnn_theorem}
    For any activation function $\sigma:\mathbb{R} \rightarrow [0,1],$ any dimension $n$ and any probability measure $\mu$ on $(\mathbb{R}^{n},\mathbb{B}^{n}), \Sigma^{n}(\sigma)$ is uniformly dense on a compact domain in $\mathcal{C}^{n}$ and $\rho_{\mu}-$dense in $\mathcal{M}^{n}.$ 
\end{theorem}

\begin{corollary}\cite{hornik1989multilayer} 
\label{corollary:ffnn_corollary}
Theorem \ref{theorem:ffnn_theorem} holds for the approximation of functions $\mathcal{C}^{m,n}$ and $\mathcal{M}^{m,n}$ by the extended function class $\Sigma^{m,n}.$ Thereby the metric $\rho_{\mu}$ is replaced by $$\rho^{n}_{\mu}:=\sum_{l=1}^{n} \rho_{\mu}(f_l,g_l).$$
\end{corollary}

We now introduce the following definition to extend the universal approximation theorem for feedforward neural networks to universal approximation of internal states of a state space system by Jordan recurrent networks. 
 
\begin{definition}
    For any Borel-measurable function $\sigma$ and $m,n \in \mathbb{N}$ define ${JRN}^{m,n}(\sigma)$ as the class of functions 
    $$x^{(t+1)}=W_{xa}\sigma(W_{ay}y^{(t)}+W_{ax}x^{(t)}-b)$$
    where $x^{(t)} \in \mathbb{R}^n, y^{(t)} \in \mathbb{R}^{m},$ for $t=1,2,...,T.$ Also, $W_{ax} \in \mathbb{R}^{\tilde{h} \times n}, W_{ay} \in \mathbb{R}^{\tilde{h} \times m}, W_{xa} \in \mathbb{R}^{n \times \tilde{h}}$ are weight matrices and $\theta \in \mathbb{R}^{\tilde{h}}$ is the corresponding bias where $\tilde{h}$ is the dimension of the hidden state.
\end{definition}

\begin{theorem}(Universal approximation theorem of state estimators using Jordan recurrent neural networks(JRNs))
    Let $F(\cdot): \mathbb{R}^{n} \times \mathbb{R}^{m} \rightarrow \mathbb{R}^{n}$ be measurable and $h(\cdot): \mathbb{R}^{n} \rightarrow \mathbb{R}^{m}$ be continuous, the states $x^{(t)} \in \mathbb{R}^{n},$ and the measurements $y^{(t)} \in \mathbb{R}^{m}$ where $t=1,2,...,T$ for finite $T$. Then, any state estimator of the form 
   $$\bar{x}^{(t+1)}=F(\bar{x}^{(t)},y^{(t)})$$
   for the discrete-time dynamical system 
\begin{equation}
    \begin{split}
        x^{(t+1)}=f(x^{(t)}) \\
        y^{(t)}=h(x^{(t)})
    \end{split}
\end{equation}
    can be approximated by an element of the function class $JRN^{m,n}(\sigma)$ with an arbitrary accuracy, where $\sigma$ is a continuous activation function.
\end{theorem}

\begin{proof}
    We want to conclude that the estimator 
    $$\bar{x}^{(t+1)}=F(\bar{x}^{(t)},y^{(t)})$$
    can be approximated by a JRN of the form 
    $$\hat{x}^{(t+1)}=W_{xa}\sigma(W_{ay}y^{(t)}+W_{ax}\hat{x}^{(t)}-b) \quad \forall t=1,2,...,T$$
    with an arbitrary accuracy, where $W_{xa},W_{ax}$ and $W_{ay}$ are appropriate weight matrices, $b$ refers to bias and $\sigma$ is the required activation function. We will begin by using the universal approximation theorem for feed-forward neural networks and then use the continuity of $F$ to reach the conclusion.

    Let $\epsilon>0$ and $K\subseteq \mathbb{R}^{n} \times \mathbb{R}^{m}$ be a compact set which contains $(\bar{x}^{(t)},y^{(t)})$ and $(\hat{x}^{(t)}, y^{(t)}), \forall t=1,2,...,T.$ From Theorem \ref{theorem:ffnn_theorem} and Corollary \ref{corollary:ffnn_corollary}, we know that for any measurable function $F(\bar{x}^{(t)},y^{(t)}): \mathbb{R}^{n} \times \mathbb{R}^{m} \rightarrow \mathbb{R}^{n}$ and for an arbitrary $\delta>0,$ a function 
    $$NN(\bar{x}^{(t)}, y^{(t)})=W_{xa}\sigma(W_{ay}y^{(t)}+W_{ax}\bar{x}^{(t)}-b)$$
    with weight matrices $W_{xa} \in \mathbb{R}^{n \times \tilde{h}}, W_{ax} \in \mathbb{R}^{\tilde{h} \times n}, W_{ay} \in \mathbb{R}^{\tilde{h} \times m},$ bias $b \in \mathbb{R}^{\tilde{h}},$ and a component-wise applied continuous activation function $\sigma(\cdot):\mathbb{R}^{\tilde{h}} \rightarrow \mathbb{R}^{\tilde{h}}$ exists, such that $\forall t=1,2,...,T,$
\begin{equation}
    \sup_{\bar{x}^{(t)},y^{(t)}} \|F(\bar{x}^{(t)},y^{(t)})-NN(\bar{x}^{(t)},y^{(t)})\|_{\infty} < \delta.
    \label{eqn:sup_norm_for_rnn}
\end{equation}
    Corresponding to a recurrent neural network architecture, $\tilde{h}$ refers to the dimension of the hidden state.

    Since $F$ is continuous, corresponding to the $\delta >0$ in Eqn.(\ref{eqn:sup_norm_for_rnn}) the following condition holds
    $$\|\bar{x}^{(t)}-\hat{x}^{(t)}\|_{\infty}<\epsilon$$
    for the following recurrent neural network architecture
    $$\hat{x}^{(t+1)}=W_{xa}\sigma(W_{ay}y^{(t)}+W_{ax}\hat{x}^{(t)}-b) \quad \forall t=1,2,....,T$$
    where $T$ is finite. This completes the proof.
\end{proof}
 Thus, JRNs are a suitable choice for state estimation.

\section{State estimation using long short-term memory networks}
While the simple recurrent neural networks described in section \ref{sec:rnns} can be trained to perform as well as a  KF for linear systems and better than the EKF for nonlinear systems (see \cite{cdc_submission}), it has been observed that for complex problems and long term predictions, long short-term memory networks are generally a better choice as they help to resolve the vanishing gradient problem (see for example, \cite{manaswi2018rnn}) and tend to converge faster. 

Each LSTM network consists of a cell. In  order to handle long term dependencies, several gates are introduced. The forget gate decides the information that should not be carried forward to the next time-step. The input gate determines the new information that needs to be remembered in the cell state. The output gate decides which information will be output from the cell state. A more detailed  explanation of these gates, for an ELSTM can be found in, for example, \cite{manaswi2018rnn}.

\begin{figure}
\centering
    \framebox{\includegraphics[scale=0.1]{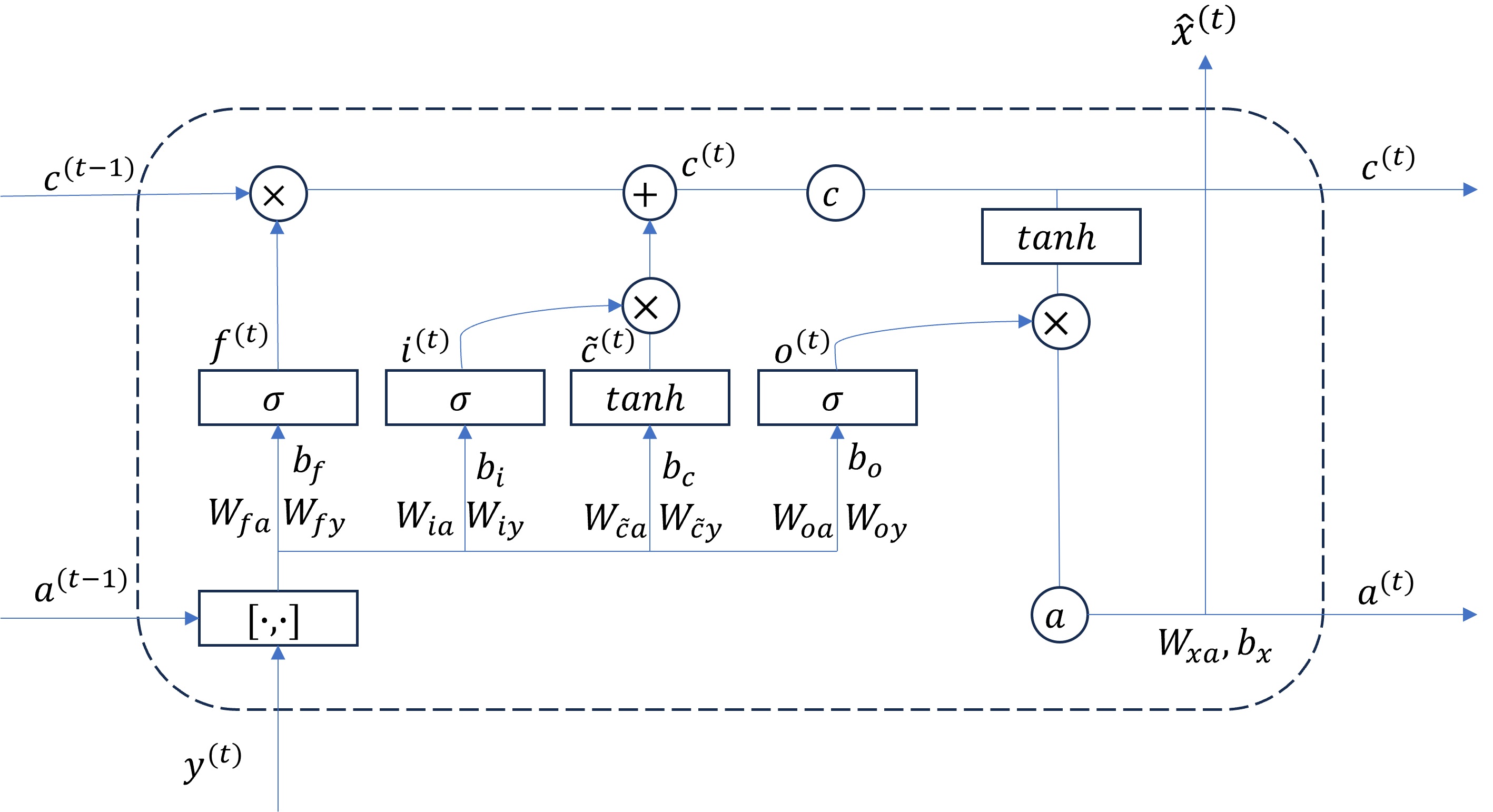}}
    \caption{Structure of an Elman long short term memory network(ELSTM) for state estimation. It uses hidden to hidden recurrent connections for state estimation of a nonlinear dynamical system. The symbols $y^{(t)}, a^{(t)}$ and $\hat{x}^{(t)}$ represent input measurement vector, hidden layer vector and estimated state vector respectively.The forget, input and output gates are represented by $f^{(t)}, i^{(t)}$ and $o^{(t)}.$ The cell state vector and the cell input activation vector are represented by $c^{(t)}$ and $\tilde{c}^{(t)}$ respectively. The weights and biases are represented by $W_{fy}, W_{fa}, W_{iy}, W_{ia}, W_{oy}, W_{oa}, W_{\tilde{c}y}, W_{\tilde{c}a}, W_{xa}, b_f, b_i, b_o$ and $b_{\tilde{c}}.$}
    \label{fig:elstm}
\end{figure}

The  forward propagation equations of ELSTM (see Fig. \ref{fig:elstm}) for state estimation are 
\begin{align}
    \begin{split}
        f^{(t)}&=\sigma(W_{fy} y^{(t)}+W_{fa} a^{(t-1)}+b_f)\\
        i^{(t)}&=\sigma(W_{iy} y^{(t)}+W_{ia} a^{(t-1)}+b_i)\\
        o^{(t)}&=\sigma(W_{oy} y^{(t)}+W_{oa} a^{(t-1)}+b_o)\\
        \tilde{c}^{(t)}&=\tanh(W_{\tilde{c}y} y^{(t)}+W_{\tilde{c}a} a^{(t-1)}+b_{\tilde{c}})\\
        c^{(t)}&=f^{(t)}\odot c^{(t-1)}+i^{(t)}\odot \tilde{c}^{(t)}\\
        a^{(t)}&=o^{(t)}\odot \tanh(c^{(t)})\\
        \hat{x}^{(t)}&=W_{xa} a^{(t)}+b_x 
    \end{split}
    \label{eqn:elstm}
\end{align}
where $f^{(t)}, i^{(t)}$ and $o^{(t)}$ are the forget, input and output gate's activation vectors respectively. Also, $\sigma(\cdot)$ is the recurrent activation function. The cell state vector and the cell input activation vector are represented by $c^{(t)}$ and $\tilde{c}^{(t)}$ respectively. The hidden layer vector is represented by $a^{(t)}.$ The weight matrices and biases are represented by $W_{fy}, W_{fa}, W_{iy}, W_{ia}, W_{oy}, W_{oa}, W_{\tilde{c}y}, W_{\tilde{c}a}, W_{xa}, b_f, b_i, b_o$ and $b_{\tilde{c}}.$ The symbol $\odot$ represents element-wise product. 

We now introduce here  a new structure called Jordan based long short-term memory network(JLSTM) for state estimation. The forward propagation equations of JLSTM (see Fig. \ref{fig:JLSTM}) for state estimation will be: 
\begin{align}
    \begin{split}
        f^{(t)}&=\sigma(W_{fy} y^{(t)}+W_{fx} \hat{x}^{(t-1)}+b_f)\\
        i^{(t)}&=\sigma(W_{iy} y^{(t)}+W_{ix} \hat{x}^{(t-1)}+b_i)\\
        o^{(t)}&=\sigma(W_{oy} y^{(t)}+W_{ox} \hat{x}^{(t-1)}+b_o)\\
        \tilde{c}^{(t)}&=\tanh(W_{\tilde{c}y} y^{(t)}+W_{\tilde{c}x} \hat{x}^{(t-1)}+b_{\tilde{c}})\\
        c^{(t)}&=f^{(t)}\odot c^{(t-1)}+i^{(t)}\odot \tilde{c}^{(t)}\\
        a^{(t)}&=o^{(t)}\odot \tanh(c^{(t)})\\
        \hat{x}^{(t)}&=W_{xa} a^{(t)}+b_x 
    \end{split}
    \label{eqn:jlstm}
\end{align}
where $f^{(t)}, i^{(t)}$ and $o^{(t)}$ are the forget, input and output gate's activation vectors respectively and function similar to the ELSTM except that they are now connected to the previous output layer instead of the previous hidden layer. The function $\sigma(\cdot)$ is the recurrent activation function. The cell state vector and the cell input activation vector are represented by $c^{(t)}$ and $\tilde{c}^{(t)}$ respectively. The hidden layer vector is represented by $a^{(t)}.$ The weight matrices and biases are represented by $W_{fy}, W_{fx}, W_{iy}, W_{ix}, W_{oy}, W_{ox}, W_{\tilde{c}y}, W_{\tilde{c}x}, W_{xa}, b_f, b_i, b_o$ and $b_{\tilde{c}}.$ The symbol $\odot$ represents element-wise product. 
\begin{figure}
\centering
    \framebox{\includegraphics[scale=0.35]{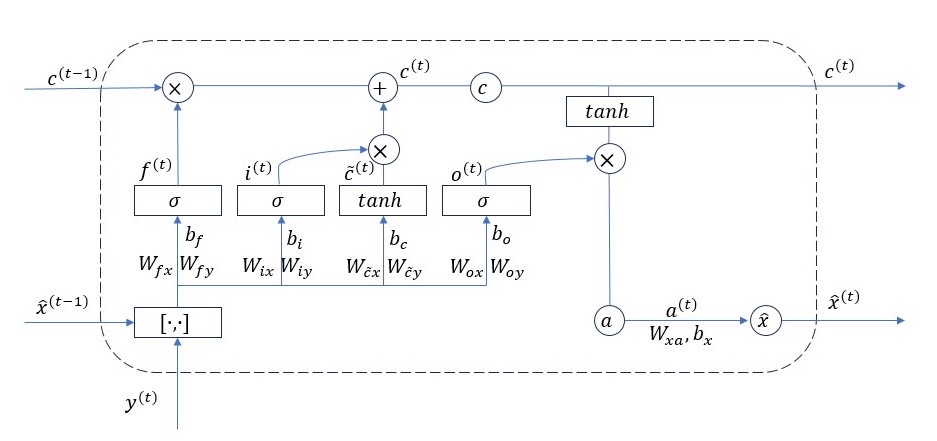}}
    \caption{Structure of a Jordan long short term memory network(JLSTM) for state estimation. It uses output to hidden recurrent connections. The symbols $y^{(t)}, a^{(t)}$ and $\hat{x}^{(t)}$ represent input measurement vector, hidden layer vector and estimated state vector respectively. The forget, input and output gates are represented by $f^{(t)}, i^{(t)}$ and $o^{(t)}.$ The cell state vector and the cell input activation vector are represented by $c^{(t)}$ and $\tilde{c}^{(t)}$ respectively. The weights and biases are represented by $W_{fy}, W_{fx}, W_{iy}, W_{ix}, W_{oy}, W_{ox}, W_{\tilde{c}y}, W_{\tilde{c}x}, W_{xa}, b_f, b_i, b_o$ and $b_{\tilde{c}}.$}
    \label{fig:JLSTM}
\end{figure}
The connections in a  JLSTM resemble the connections in a JRN. In particular, the previous output $\hat{x}^{(t-1)}$ is fed back into the network at the next time-step as can be seen in \eqref{eqn:jlstm} and Fig. \ref{fig:JLSTM}.

Note that simple recurrent neural networks are a special case of long short term memory networks where the gates are not involved. The introduction of $f^{(t)}, i^{(t)}, o^{(t)}, \tilde{c}^{(t)}$ and $c^{(t)}$ in \eqref{eqn:elstm}, \eqref{eqn:jlstm}, Fig. \ref{fig:elstm} and Fig. \ref{fig:JLSTM} is to handle long term dependencies better. Thus, the universal approximation theorems discussed in the previous section are valid. 

\section{Implementation of ELSTM and JLSTM}
\label{sec:implementation}

The fact that the functions $f$ and $h$ are known in \eqref{eqn:noisy_disc_state_space_sys} was used in data generation.
We considered sequences starting from time $0.$  Gaussian initial conditions were sampled over an interval with covariance of $(0.01) \times I$ where $I$ is an identity matrix of appropriate dimensions. Gaussian process and measurement noises with zero mean and known covariance were sampled for each time-step. This was done for each sequence in the data set. Measurement and process noise covariance were considered to be $(0.01) \times I$ where the  matrix $I$ is of appropriate dimensions. The number of hidden units considered were $50$ for both networks.

A total of $100$ sequences were generated. These were divided into three parts: training, validation and testing sequences in the ratio $80:10:10.$ 

ELSTM and JLSTM with forward propagation as in \eqref{eqn:elstm} and \eqref{eqn:jlstm} respectively are implemented. The weight matrices $W_{fy}, W_{iy}, W_{oy}, W_{\tilde{c}y}, W_{fa}, W_{ia}, W_{oa}, W_{\tilde{c}a}, W_{fx}, W_{ix}, W_{ox}$ and $W_{\tilde{c}x}$ are initialised using Glorot uniform distribution while the weight matrix $W_{xa}$ is initialised to be an orthogonal matrix using \textit{Tensorflow}'s $tensorflow.keras.initializers$ module. The recurrent activation function $\sigma(z)$ is chosen to be to be $sigmoid(z).$ The backward propagation is implemented using \textit{Tensorflow}'s $apply\_gradients$ module. 

The loss function at time-step $t$ for each sequence is 
\begin{equation}
    J_{i}^{(t)}(\phi)=\frac{1}{n}\sum_{i=1}^{n}{(x^{(t)}_{i}-\hat{x}^{(t)}_{i,\phi})}^{2}, 
\end{equation}
where $1\leq i\leq n$ and $n$ is the number of states, $x^{(t)}_{i}$ represents the true value of state $x_i$ at time-step $t$ and $\hat{x}^{(t)}_{i,\phi}$ represents the estimated value of state $x_i$ at time-step $t$ by the respective network where $\phi$ denotes hyper-parameters of the network. This is called the minimum variance cost function and was considered as we want the estimated state to be as close to the true state as possible. 

Adam optimization was used for training the networks. The number of hidden units and the maximum number of epochs were kept same for a particular example for both networks to make sure the comparison is accurate.
Early stopping with a fixed patience value(the same for both networks) to decide the number of epochs needed is used. In each epoch, the network's performance on validation data was used as the measure to decide whether to proceed or keep training for another epoch.

We tested JLSTM for 3 different examples in Section~\ref{sec:numerical_examples} and compared the results with ELSTM and KF for the linear system and EKF for the nonlinear systems. Standard implementations of KF and EKF are used, see for example \cite{dansimon}. 

Testing was done with two sets of initial conditions.   We first test the estimators using 10\% of the data set reserved for testing.  We then test using data generated from a random Gaussian initial condition with a mean different from that of the first set.  

We compare JLSTM, ELSTM and KF/EKF using average error at time $t$ over all features and all test sequences
\begin{equation}
    \text{Error}(t)=\frac{1}{m_{test} n} \sum_{k=1}^{m_{test}} \sum_{i=1}^{n} (x_i^{(t)[k]}-\hat{x}_i^{(t)[k]})^2,
\end{equation}
where $m_{test}$ is the number of test sequences, $n$ is the number of states, $x_i^{(t)[k]}$ represents the true value of state $x_{i}$ at time-step $t$ for the $k^{th}$ sequence and $\hat{x}_i^{(t)[k]}$ represents the estimated value of state $x_{i}$ at time-step $t$ for the $k^{th}$ sequence.
We also used normalised mean square error (NMSE) to compare different methods. It is calculated as
\begin{equation}
    \text{NMSE}=\frac{1}{m_{test} n T} \sum_{k=1}^{m_{test}} \sum_{j=1}^{T} \sum_{i=1}^{n} (x_i^{(j)[k]}-\hat{x}_i^{(j)[k]})^2,
\end{equation}
where $m_{test}$ is the number of test sequences, $T$ is the number of time-steps in each sequence, $n$ is the number of states, $x_{i}^{(j)[k]}$ represents the true value of state $x_{i}$ at time-step $j$ for the $k^{th}$ sequence and $\hat{x}_{i}^{(j)[k]}$ represents the estimated value of state $x_{i}$ at time-step $j$ for the $k^{th}$ sequence. 

\textit{Tensorflow} was used for the implementation of both networks. Codes were run on a T4 GPU.

\section{Examples}
\label{sec:numerical_examples}
In this section, we compare the proposed method with existing methods for several examples. Discrete-time state-space models can arise as discretisations of a system of ordinary differential equations used for estimating state of a system equipped with measurements. Several methods exist to discretise a continuous-time system of ordinary differential equations. We used zero order hold discretisation for the linear system and an RK-45 discretisation for the nonlinear systems of ODEs.

\subsection{Linear system: Connected springs}\label{LS}
We considered a zero-order hold time discretization with $\Delta t=0.1s$ of a system  of $10$ connected springs
\begin{equation}
\begin{split}
m_1 \ddot{x}_1(t) &= -k_1 x_1(t) + k_2 (x_2(t) - x_1(t)) \\
&\quad - d_1 \dot{x}_1(t) + d_2 (\dot{x}_2(t) - \dot{x}_1(t))\\
m_2 \ddot{x}_1(t) &= -k_2 (x_2(t) - x_1(t)) + k_3 (x_3(t) - x_2(t)) \\
&\quad - d_2 (\dot{x}_2(t) - \dot{x}_1(t)) + d_3 (\dot{x}_3(t) - \dot{x}_2(t))\\
\vdots \\
m_{9} \ddot{x}_{9}(t) &= -k_{9} (x_{9}(t) - x_{8}(t)) + k_{10} (x_{10}(t) - x_{9}(t)) \\
&\quad - d_{9} (\dot{x}_{9}(t) - \dot{x}_{8}(t)) + d_{10} (\dot{x}_{10}(t) - \dot{x}_{9}(t))\\
m_{10} \ddot{x}_{10}(t) &= -k_{10} (x_{10}(t) - x_{9}(t)) - d_{10} (\dot{x}_{10}(t) - \dot{x}_{9}(t))
\end{split}
\label{eq:connected_springs}
\end{equation}
where $m_i = 10, d_i = 6$ and $k_i = 800$ for all $i.$ The state vector consists of $20$ discretised states, the position $x_i(t)$ and velocity $\dot{x}_i(t)$ for each of the $10$ springs. The process noise is added to each state after discretization. The discretised measurement vector at time-step $t$ is given by 
$$y^{(t)}=
\begin{bmatrix}
    x_1^{(t)}\\ x_2^{(t)}\\ \vdots \\ x_{10}^{(t)}
\end{bmatrix}+\nu^{(t)}$$
where $x_i^{(t)}$ is the discretised state representing position of spring $i$ and $\nu^{(t)}$ is the measurement noise at time-step $t$. 

We considered sequences of length $500$ starting at time $0$ seconds. Gaussian initial conditions for each sequence were chosen randomly from $[-1,1]^{20}$ with a covariance of $(0.01) \times I_{20}$ where $I_{k}$ represents an identity matrix of order $k.$ All noises were assumed to be Gaussian with zero mean and known covariance. Measurement and process noise covariance were considered to be $(0.01) \times I$ where matrix $I$ is of appropriate dimensions. A total of $100$ sequences were considered with a batch size of $10.$ The number of hidden units considered were $50.$ Early stopping was implemented with a patience of $50$. The maximum number of epochs considered was $8000.$ Training using Adam optimization with a learning rate of $10^{-3}$ for ELSTM as well as JLSTM achieved a close NMSE value for both networks. 

\begin{figure}
\centering
    \framebox{\includegraphics[scale=0.38]{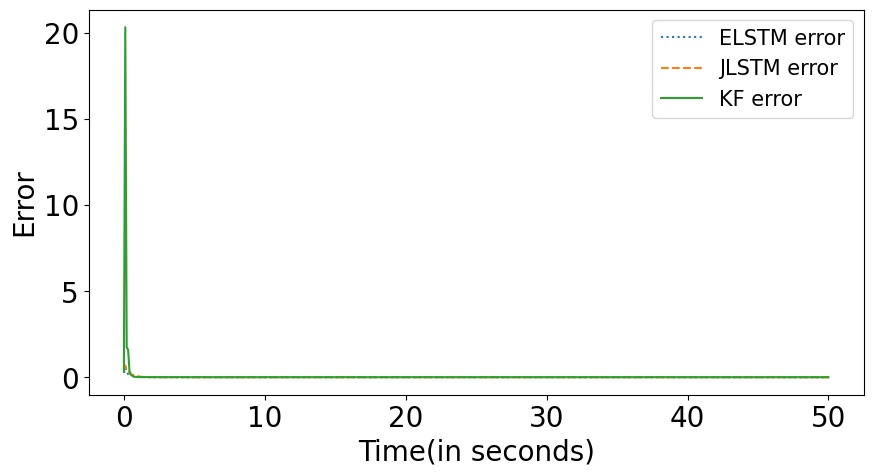}}
    \caption{This figure compares the performance of KF, JLSTM and ELSTM using the average errors over $10$ test sequences for $50$ seconds of $10$ connected springs with a noisy Gaussian initial condition.}
    \label{fig:connected_springs}
\end{figure}

While the ELSTM and JLSTM have almost equal errors at all time-steps (see Fig. \ref{fig:connected_springs}), the KF has the largest error at the initial time-step but converges quickly.

\begin{figure}
\centering
    \framebox{\includegraphics[scale=0.38]{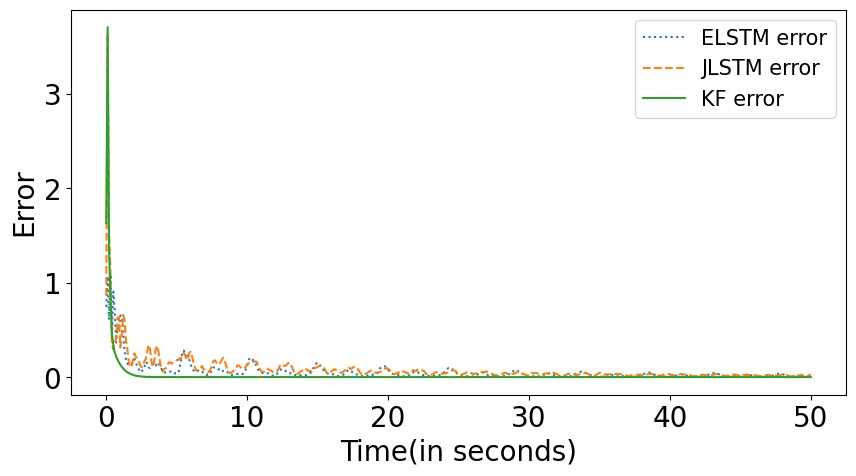}}
    \caption{This figure compares the performance of KF, JLSTM and ELSTM using the average errors over $10$ test sequences for $50$ seconds of $10$ connected springs with a noisy Gaussian initial condition outside the training region.}
    \label{fig:connected_springs_error_OR}
\end{figure}

For testing outside the initial interval of mean of initial condition, in this example, we used the interval ${[1,1.5]}^{20} \sim \mathbb{R}^{20}.$ As can be seen from Fig. \ref{fig:connected_springs_error_OR}, the ELSTM and JLSTM converge slower than the KF. This is not surprising as theory shows the KF to be the best estimator for a linear system. However, the closeness of the errors of the ELSTM and JLSTM validates the training of the networks. To keep the comparison fair, for outside the range testing, the KF was not changed.

\subsection{Nonlinear system: down pendulum}
We considered an RK-45 discretization with $\Delta t=0.01s$ of the down pendulum 
\begin{equation}
        \boldsymbol{\dot{x}}(t)=
        \begin{bmatrix}
            \dot{x_1}(t)\\ \dot{x_2}(t)
        \end{bmatrix}
            =
        \begin{bmatrix}
            x_2(t)\\
            -\frac{g}{l} \sin({x_1(t)})-\frac{b}{m} x_2(t)
        \end{bmatrix},
    \label{eq:pendulum}
\end{equation}
where $x_1(t)$ and $x_2(t)$ are the states representing position and velocity of the pendulum respectively. The parameters are $g=9.8 m/{sec}^2,$  $m=2 kg,$ $b=0.9 kg/sec$ and  $l=1 m.$ The process noise is added to each state after discretization. The discretised measurement vector at time-step $t$ is $y^{(t)}=x_1^{(t)}+\nu^{(t)}$ where $x_1{(t)}$ is the discretised state representing position of the pendulum and $\nu^{(t)}$ is the measurement noise at time-step $t$. 
 
We considered sequences of length $4000$ starting at time $0$ seconds. Gaussian initial conditions for each sequence were chosen randomly from $[-2,2] \times [-2,2]$ with an initial condition covariance of $(0.01) \times I_{2}$ where $I_{k}$ represents an identity matrix of order $k.$ All noises were assumed to be Gaussian with zero mean and known covariance. Measurement and process noise covariance were considered to be $(0.01) \times I$ where matrix $I$ is of appropriate dimensions. A total of $100$ sequences were considered with a batch size of $20.$ The number of hidden units considered were $50.$ Early stopping was implemented with a patience of $50$. The maximum number of epochs considered was $3000.$ Training using Adam optimization with a learning rate of $10^{-4}$ for ELSTM and $10^{-3}$ for JLSTM achieved a close NMSE value for both networks. 

\begin{figure}
\centering
    \framebox{\includegraphics[scale=0.37]{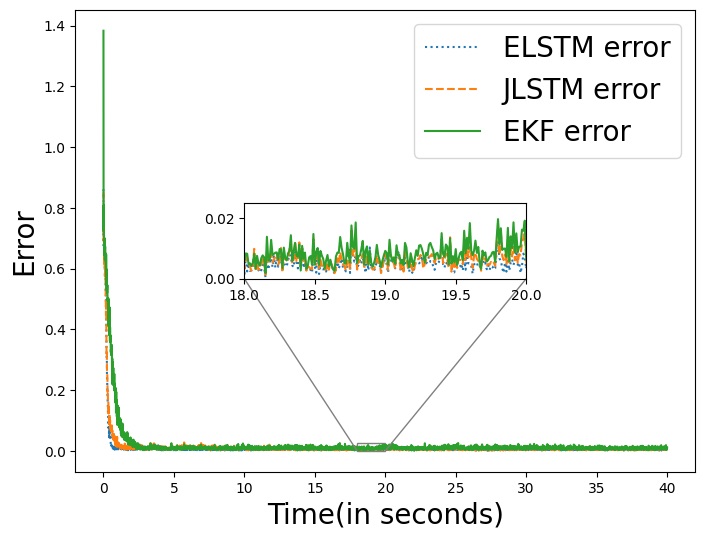}}
    \caption{This figure compares the performance of EKF, JLSTM and ELSTM using the normalised mean squared errors over $10$ test sequences for $40$ seconds of down pendulum with a noisy Gaussian initial condition.}
    \label{fig:pend_error}
\end{figure}

While the ELSTM and JLSTM have almost equal errors at all time-steps(see Fig. \ref{fig:pend_error}), the EKF has the largest error at the initial time-step and converges later than the other two. 

\begin{figure}
\centering
    \framebox{\includegraphics[scale=0.46]{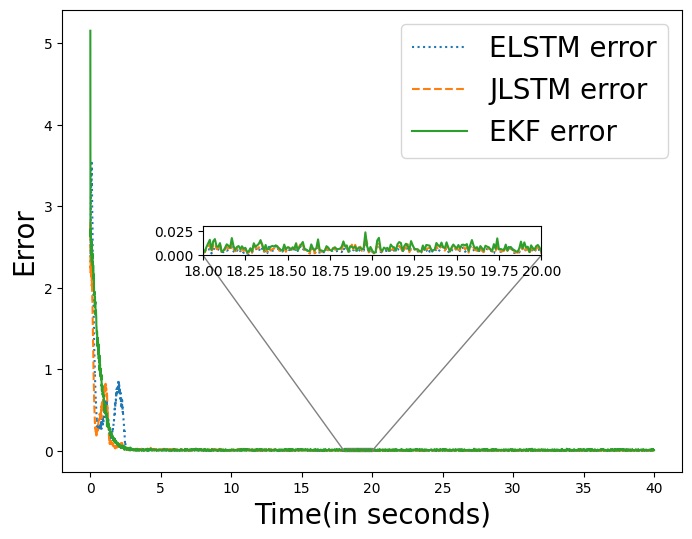}}
    \caption{This figure compares the performance of EKF, JLSTM and ELSTM using the normalised mean squared errors over $10$ test sequences for $40$ seconds of down pendulum with a noisy Gaussian initial condition outside the training region.}
    \label{fig:pend_error_OR}
\end{figure}

For testing outside the initial interval of mean of initial condition, in this example, we used the interval ${[2,2.5]}^2 \sim \mathbb{R}^2.$ As can be seen from Fig. \ref{fig:pend_error_OR}, the ELSTM and JLSTM perform better than the EKF especially at the initial time-steps. To keep the comparison fair, for outside the range testing, the EKF was not changed.

\subsection{Nonlinear system: reversed Van der Pol oscillator}
We considered an RK-45 discretization with $\Delta t=0.1s$ of the reversed Van der Pol oscillator 
\begin{equation}
        \boldsymbol{\dot{x}}(t)=
        \begin{bmatrix}
            \dot{x_1}(t)\\ \dot{x_2}(t)
        \end{bmatrix}
            = 
        \begin{bmatrix}
            -x_2(t)\\
            x_1(t)+(({x_1}(t))^2-1)x_2(t)
        \end{bmatrix}
    \label{eq:oscillator}
\end{equation}
where $x_1(t)$ and $x_2(t)$ is the position and velocity of the oscillator. The discretised measurement vector at time-step $t$ is $y^{(t)}=x_1^{(t)}+\nu^{(t)}$ where $\nu^{(t)}$ is the measurement noise at time-step $t$.

For this example, we have considered sequences of length $300$ starting from time $0.$ The Gaussian initial condition was sampled from $[-1, 1] \times [-1, 1]$ with covariance $(0.01) \times I_{2}$ where $I_{k}$ represents an identity matrix of order $k$ for each sequence. All noises were assumed to be Gaussian with zero mean and known covariance. Measurement and process noise covariance were again  $(0.01) \times I$ where matrix $I$ is of appropriate dimensions. A total of $100$ sequences were considered with a batch size of $20$. The number of hidden units considered were $50.$ Early stopping was implemented with a patience of $15$. The maximum number of epochs was set to be $3000.$ Training using Adam optimization with a learning rate of $10^{-3}$ for ELSTM and $10^{-2}$ for JLSTM achieves a close NMSE for both networks. 

\begin{figure}
\centering
    \framebox{\includegraphics[scale=0.45]{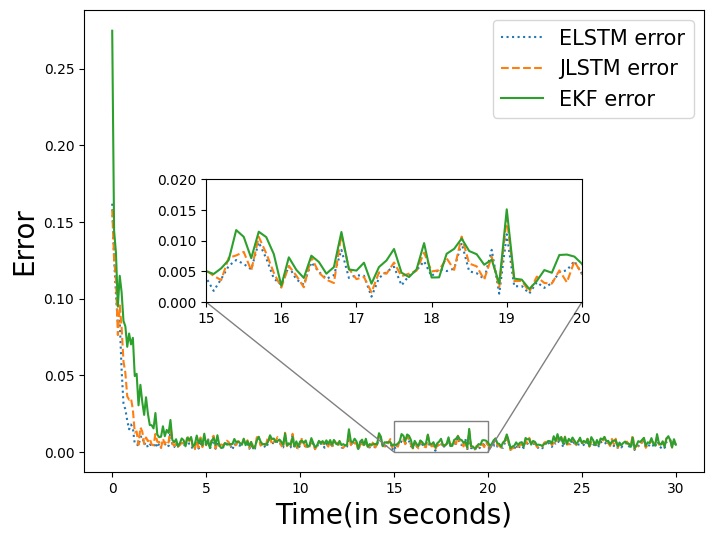}}
    \caption{This figure compares the performance of EKF, JLSTM and ELSTM using the average errors over $10$ test sequences for $30$ seconds of reversed Van der Pol oscillator with a noisy Gaussian initial condition.}
    \label{fig:reversed_van_der_pol_error}
\end{figure}

While the ELSTM and JLSTM have almost equal errors at all time-steps(see Fig. \ref{fig:reversed_van_der_pol_error}), the EKF has the largest error at the initial time-step and converges later than the other two estimators. This is  similar to the case of the  down pendulum. 

\begin{figure}
\centering
    \framebox{\includegraphics[scale=0.45]{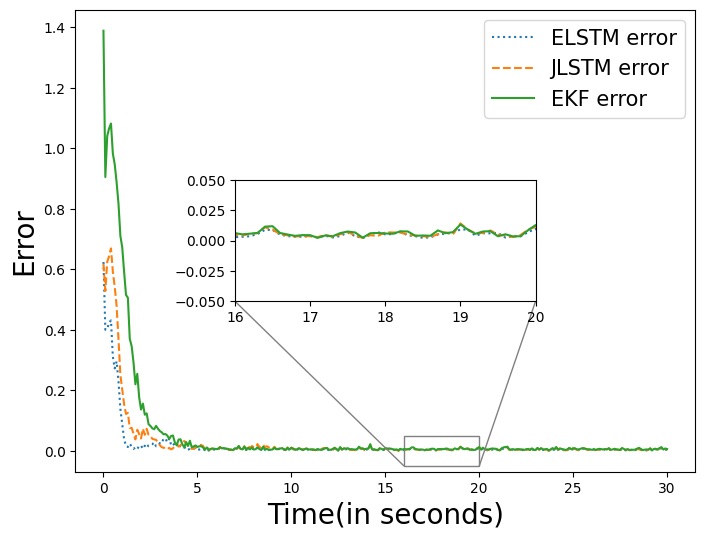}}
    \caption{This figure compares the performance of EKF, JLSTM and ELSTM using the average errors over $10$ test sequences for the first $30$ seconds of reversed Van der Pol oscillator with a noisy Gaussian initial condition outside the training region.}
    \label{fig:reversed_van_der_pol_error_OR}
\end{figure}

For testing outside the initial interval of mean of initial condition, in this example, we used the interval ${[1,1.5]}^2 \sim \mathbb{R}^2.$ As can be seen from Fig. \ref{fig:reversed_van_der_pol_error_OR}, the ELSTM and JLSTM perform better than the EKF especially at the initial time-steps. Again, for outside the range testing,  the EKF was not changed.  

\section{Results and conclusions}
\label{sec:results}

We have compared JLSTM, ELSTM and KF for a 100 dimensional linear system  and the JLSTM, ELSTM and EKF for two  nonlinear systems of order 2, the   down pendulum and the reversed Van der Pol oscillator.

\begin{table}
\caption{Normalised mean square error(NMSE) for E(KF), JLSTM and ELSTM}
\begin{center}
    \begin{tabular}{|c|c|c|c|c|c|}
        \hline
        \textbf{State space} &\multicolumn{3}{|c|}{\textbf{Estimator}}\\
        \cline{2-4}
        \textbf{system} & \textbf{\textit{E(KF)}} & \textbf{\textit{JLSTM}} & \textbf{\textit{ELSTM}}\\
        \hline
        Connected springs & 0.0174 & 0.0162 & 0.0153\\
        Down pendulum & 0.0695 & 0.0542 & 0.0461 \\
        r. Van der Pol & 0.0635 & 0.0485 & 0.0441 \\
        \hline
    \end{tabular}
    \label{tab:MSE}
    \end{center}
\end{table}

The results in Table \ref{tab:MSE} indicate that the NMSE value corresponding to JLSTM and ELSTM is close to the KF's NMSE value for the linear system. 

For both nonlinear systems, both JLSTM and ELSTM have a smaller NMSE than the EKF. This is likely  because there is less linearization error. Thus, the presence of a nonlinear activation function appears advantageous.  

\begin{table}
\caption{Training time(in seconds) for JLSTM and ELSTM}
\begin{center}
    \begin{tabular}{|c|c|c|c|c|}
        \hline
        \textbf{State space} &\multicolumn{2}{|c|}{\textbf{Estimator}}\\
        \cline{2-3}
        \textbf{system} & \textbf{\textit{JLSTM}} &  \textbf{\textit{ELSTM}}\\
        \hline
        Connected springs & 15983 & 16298 \\
        Down pendulum & 95644 & 99473 \\
        r. Van der Pol & 619 & 3389 \\
    \hline
    \end{tabular}
    \label{tab:time_to_train}
    \end{center}
\end{table}

Because the  learning rate was chosen so that the error values for ELSTM and JLSTM became comparable for all 3 examples, a comparison of their training times is relevant. The time taken to train an ELSTM was considerably larger than the time taken to train a JLSTM (Table \ref{tab:time_to_train}). This suggests that having direct connections with state estimates at the previous time-step helps to minimise the loss function faster.

\begin{table}
\caption{Testing time(in seconds) for E(KF), JLSTM and ELSTM}
\begin{center}
    \begin{tabular}{|c|c|c|c|c|c|}
        \hline
        \textbf{State space} &\multicolumn{3}{|c|}{\textbf{Estimator}}\\
        \cline{2-4}
        \textbf{system} & \textbf{\textit{E(KF)}} & \textbf{\textit{JLSTM}} &  \textbf{\textit{ELSTM}}\\
        \hline
        Connected springs & 0.8 & 0.3 & 0.3 \\
        Down pendulum & 38.2 & 2.7 & 2.7 \\
        r. Van der Pol & 3.2 & 0.2 & 0.2 \\
    \hline
    \end{tabular}
    \label{tab:time_to_test}
    \end{center}
\end{table}

While the time taken to train both the JLSTM and ELSTM is large as can be seen in Table \ref{tab:time_to_train}, the testing time only involves simple operations and is quite fast. We show their comparison for time taken to test $10$ sequences for each example in Table \ref{tab:time_to_test}. Both the LSTM architectures are faster than the KF and the EKF in the respective cases. The execution time for both the ELSTM and JLSTM is almost the same.

\begin{table}
\caption{Normalised mean square error(NMSE) for E(KF), JLSTM and ELSTM for initial conditions outside the training region}
\begin{center}
    \begin{tabular}{|c|c|c|c|c|c|}
        \hline
        \textbf{State space} &\multicolumn{3}{|c|}{\textbf{Estimator}}\\
        \cline{2-4}
        \textbf{system} & \textbf{\textit{E(KF)}} & \textbf{\textit{JLSTM}} & \textbf{\textit{ELSTM}}\\
        \hline
        Connected springs & 0.0179 & 0.0567 & 0.0459\\
        Down pendulum & 0.0850 & 0.0687 & 0.0700 \\
        r. Van der Pol & 0.0927 & 0.0586 & 0.0497 \\
        \hline
    \end{tabular}
    \label{tab:MSE_OR}
    \end{center}
\end{table}

When tested on data generated using initial conditions outside the training range, we observe from Table \ref{tab:MSE_OR} that the NMSE for JLSTM and ELSTM is much better than that for EKF for non-linear examples but more than the NMSE for KF for the linear example. It does appear from this that the networks are learning the model better for the nonlinear examples. 

We thus conclude that both  JLSTM and ELSTM provide smaller errors than an EKF for nonlinear systems. Using a JLSTM instead of an ELSTM  appears preferable because the JLSTM network has considerably smaller training time to achieve the same error.

\section{Future Work}

Given the success of JLSTM on low-order nonlinear systems, extending this approach to  high-order nonlinear systems is of value. A framework to speed up back-propagation which will in turn speed up the training process will be investigated. 

One advantage of the Jordan structure is that it mimics the structure of a dynamical system. This facilitates obtaining stability results. Some preliminary results in this direction are reported in \cite{cdc_submission} for a simpler Jordan architecture. An important goal is to establish stability and convergence properties of the JLSTM. 

\bibliographystyle{apacite} 
\bibliography{references_se.bib}

\begin{thebibliography}{}

\bibitem [\protect \citeauthoryear {%
Adhyaru%
}{%
Adhyaru%
}{%
{\protect \APACyear {2012}}%
}]{%
adhyaru}
\APACinsertmetastar {%
adhyaru}%
\begin{APACrefauthors}%
Adhyaru, D\BPBI M.%
\end{APACrefauthors}%
\unskip\
\newblock
\APACrefYearMonthDay{2012}{}{}.
\newblock
{\BBOQ}\APACrefatitle {State observer design for nonlinear systems using neural network} {State observer design for nonlinear systems using neural network}.{\BBCQ}
\newblock
\APACjournalVolNumPages{Applied Soft Computing}{12}{8}{2530-2537}.
\PrintBackRefs{\CurrentBib}

\bibitem [\protect \citeauthoryear {%
Afshar%
, Germ%
\BCBL {}\ \BBA {} Morris%
}{%
Afshar%
\ \protect \BOthers {.}}{%
{\protect \APACyear {2023}}%
}]{%
AfsharGermMorris}
\APACinsertmetastar {%
AfsharGermMorris}%
\begin{APACrefauthors}%
Afshar, S.%
, Germ, F.%
\BCBL {}\ \BBA {} Morris, K.%
\end{APACrefauthors}%
\unskip\
\newblock
\APACrefYearMonthDay{2023}{}{}.
\newblock
{\BBOQ}\APACrefatitle {Extended {K}alman filter based observer design for semilinear infinite-dimensional systems} {Extended {K}alman filter based observer design for semilinear infinite-dimensional systems}.{\BBCQ}
\newblock
\APACjournalVolNumPages{IEEE Transactions on Automatic Control}{}{}{}.
\PrintBackRefs{\CurrentBib}

\bibitem [\protect \citeauthoryear {%
Benosman%
\ \BBA {} Borggaard%
}{%
Benosman%
\ \BBA {} Borggaard%
}{%
{\protect \APACyear {2021}}%
{\protect \APACexlab {{\protect \BCnt {1}}}}}]{%
benosman_borggaard_red_order_2}
\APACinsertmetastar {%
benosman_borggaard_red_order_2}%
\begin{APACrefauthors}%
Benosman, M.%
\BCBT {}\ \BBA {} Borggaard, J.%
\end{APACrefauthors}%
\unskip\
\newblock
\APACrefYearMonthDay{2021{\protect \BCnt {1}}}{}{}.
\newblock
{\BBOQ}\APACrefatitle {Data-driven robust state estimation for reduced-order models of 2d {B}oussinesq equations with parametric uncertainties} {Data-driven robust state estimation for reduced-order models of 2d {B}oussinesq equations with parametric uncertainties}.{\BBCQ}
\newblock
\APACjournalVolNumPages{Computers and Fluids}{214}{}{104773}.
\PrintBackRefs{\CurrentBib}

\bibitem [\protect \citeauthoryear {%
Benosman%
\ \BBA {} Borggaard%
}{%
Benosman%
\ \BBA {} Borggaard%
}{%
{\protect \APACyear {2021}}%
{\protect \APACexlab {{\protect \BCnt {2}}}}}]{%
benosman_borggaard_red_order}
\APACinsertmetastar {%
benosman_borggaard_red_order}%
\begin{APACrefauthors}%
Benosman, M.%
\BCBT {}\ \BBA {} Borggaard, J.%
\end{APACrefauthors}%
\unskip\
\newblock
\APACrefYearMonthDay{2021{\protect \BCnt {2}}}{}{}.
\newblock
{\BBOQ}\APACrefatitle {Robust nonlinear state estimation for a class of infinite-dimensional systems using reduced-order models} {Robust nonlinear state estimation for a class of infinite-dimensional systems using reduced-order models}.{\BBCQ}
\newblock
\APACjournalVolNumPages{International Journal of Control}{94}{}{1309-1320}.
\PrintBackRefs{\CurrentBib}

\bibitem [\protect \citeauthoryear {%
Bernard%
, Andrieu%
\BCBL {}\ \BBA {} Astolfi%
}{%
Bernard%
\ \protect \BOthers {.}}{%
{\protect \APACyear {2022}}%
}]{%
bernard2022observer}
\APACinsertmetastar {%
bernard2022observer}%
\begin{APACrefauthors}%
Bernard, P.%
, Andrieu, V.%
\BCBL {}\ \BBA {} Astolfi, D.%
\end{APACrefauthors}%
\unskip\
\newblock
\APACrefYearMonthDay{2022}{}{}.
\newblock
{\BBOQ}\APACrefatitle {Observer design for continuous-time dynamical systems} {Observer design for continuous-time dynamical systems}.{\BBCQ}
\newblock
\APACjournalVolNumPages{Annual Reviews in Control}{}{}{}.
\PrintBackRefs{\CurrentBib}

\bibitem [\protect \citeauthoryear {%
Bertsekas%
}{%
Bertsekas%
}{%
{\protect \APACyear {2019}}%
}]{%
rl_book}
\APACinsertmetastar {%
rl_book}%
\begin{APACrefauthors}%
Bertsekas, D\BPBI P.%
\end{APACrefauthors}%
\unskip\
\newblock
\APACrefYear{2019}.
\newblock
\APACrefbtitle {Reinforcement learning and optimal control} {Reinforcement learning and optimal control}.
\newblock
\APACaddressPublisher{}{Athena Scientific}.
\PrintBackRefs{\CurrentBib}

\bibitem [\protect \citeauthoryear {%
Breiten%
\ \BBA {} Kunisch%
}{%
Breiten%
\ \BBA {} Kunisch%
}{%
{\protect \APACyear {2021}}%
}]{%
kunisch}
\APACinsertmetastar {%
kunisch}%
\begin{APACrefauthors}%
Breiten, T.%
\BCBT {}\ \BBA {} Kunisch, K.%
\end{APACrefauthors}%
\unskip\
\newblock
\APACrefYearMonthDay{2021}{}{}.
\newblock
{\BBOQ}\APACrefatitle {Neural network based nonlinear observers} {Neural network based nonlinear observers}.{\BBCQ}
\newblock
\APACjournalVolNumPages{Systems \& Control Letters}{148}{}{104829}.
\PrintBackRefs{\CurrentBib}

\bibitem [\protect \citeauthoryear {%
Cana%
, Herrero%
\BCBL {}\ \BBA {} Lopez%
}{%
Cana%
\ \protect \BOthers {.}}{%
{\protect \APACyear {2021}}%
}]{%
llerena2021forecasting}
\APACinsertmetastar {%
llerena2021forecasting}%
\begin{APACrefauthors}%
Cana, J\BPBI P\BPBI L.%
, Herrero, J\BPBI G.%
\BCBL {}\ \BBA {} Lopez, J\BPBI M\BPBI M.%
\end{APACrefauthors}%
\unskip\
\newblock
\APACrefYearMonthDay{2021}{}{}.
\newblock
{\BBOQ}\APACrefatitle {Forecasting nonlinear systems with {LSTM}: analysis and comparison with {EKF}} {Forecasting nonlinear systems with {LSTM}: analysis and comparison with {EKF}}.{\BBCQ}
\newblock
\APACjournalVolNumPages{Sensors}{21}{5}{1805}.
\PrintBackRefs{\CurrentBib}

\bibitem [\protect \citeauthoryear {%
da C.~Ramos%
, Meglio%
, Morgenthaler%
, da Silva%
\BCBL {}\ \BBA {} Bernard%
}{%
da C.~Ramos%
\ \protect \BOthers {.}}{%
{\protect \APACyear {2020}}%
}]{%
ramos}
\APACinsertmetastar {%
ramos}%
\begin{APACrefauthors}%
da C.~Ramos, L.%
, Meglio, F\BPBI D.%
, Morgenthaler, V.%
, da Silva, L\BPBI F\BPBI F.%
\BCBL {}\ \BBA {} Bernard, P.%
\end{APACrefauthors}%
\unskip\
\newblock
\APACrefYearMonthDay{2020}{}{}.
\newblock
{\BBOQ}\APACrefatitle {Numerical design of {L}uenberger observers for nonlinear systems} {Numerical design of {L}uenberger observers for nonlinear systems}.{\BBCQ}
\newblock
\APACjournalVolNumPages{59th IEEE Conference on Decision and Control}{}{}{5435-5442}.
\PrintBackRefs{\CurrentBib}

\bibitem [\protect \citeauthoryear {%
Fu%
, Xie%
\BCBL {}\ \BBA {} Na%
}{%
Fu%
\ \protect \BOthers {.}}{%
{\protect \APACyear {2016}}%
}]{%
zhijunfu}
\APACinsertmetastar {%
zhijunfu}%
\begin{APACrefauthors}%
Fu, Z\BHBI J.%
, Xie, W\BHBI F.%
\BCBL {}\ \BBA {} Na, J.%
\end{APACrefauthors}%
\unskip\
\newblock
\APACrefYearMonthDay{2016}{}{}.
\newblock
{\BBOQ}\APACrefatitle {Robust adaptive nonlinear observer design via multi-time scales neural network} {Robust adaptive nonlinear observer design via multi-time scales neural network}.{\BBCQ}
\newblock
\APACjournalVolNumPages{Neurocomputing}{190}{}{217-225}.
\PrintBackRefs{\CurrentBib}

\bibitem [\protect \citeauthoryear {%
C.~Gao%
, Yan%
, Zhou%
, Varshney%
\BCBL {}\ \BBA {} Liu%
}{%
C.~Gao%
\ \protect \BOthers {.}}{%
{\protect \APACyear {2019}}%
}]{%
gao_2019}
\APACinsertmetastar {%
gao_2019}%
\begin{APACrefauthors}%
Gao, C.%
, Yan, J.%
, Zhou, S.%
, Varshney, P\BPBI K.%
\BCBL {}\ \BBA {} Liu, H.%
\end{APACrefauthors}%
\unskip\
\newblock
\APACrefYearMonthDay{2019}{}{}.
\newblock
{\BBOQ}\APACrefatitle {Long short-term memory-based deep recurrent neural networks for target tracking} {Long short-term memory-based deep recurrent neural networks for target tracking}.{\BBCQ}
\newblock
\APACjournalVolNumPages{Information Sciences}{502}{}{279-296}.
\PrintBackRefs{\CurrentBib}

\bibitem [\protect \citeauthoryear {%
X.~Gao%
\ \protect \BOthers {.}}{%
X.~Gao%
\ \protect \BOthers {.}}{%
{\protect \APACyear {2020}}%
}]{%
rl_paper}
\APACinsertmetastar {%
rl_paper}%
\begin{APACrefauthors}%
Gao, X.%
, Luo, H.%
, Ning, B.%
, Zhao, F.%
, Bao, L.%
, Gong, Y.%
\BDBL {}Jiang, J.%
\end{APACrefauthors}%
\unskip\
\newblock
\APACrefYearMonthDay{2020}{}{}.
\newblock
{\BBOQ}\APACrefatitle {{RL-AKF}: An adaptive {K}alman filter navigation algorithm based on reinforcement learning for ground vehicles} {{RL-AKF}: An adaptive {K}alman filter navigation algorithm based on reinforcement learning for ground vehicles}.{\BBCQ}
\newblock
\APACjournalVolNumPages{Remote Sensing}{12}{11}{1704}.
\PrintBackRefs{\CurrentBib}

\bibitem [\protect \citeauthoryear {%
Gencay%
\ \BBA {} Liu%
}{%
Gencay%
\ \BBA {} Liu%
}{%
{\protect \APACyear {1997}}%
}]{%
gencay1997nonlinear}
\APACinsertmetastar {%
gencay1997nonlinear}%
\begin{APACrefauthors}%
Gencay, R.%
\BCBT {}\ \BBA {} Liu, T.%
\end{APACrefauthors}%
\unskip\
\newblock
\APACrefYearMonthDay{1997}{}{}.
\newblock
{\BBOQ}\APACrefatitle {Nonlinear modelling and prediction with feedforward and recurrent networks} {Nonlinear modelling and prediction with feedforward and recurrent networks}.{\BBCQ}
\newblock
\APACjournalVolNumPages{Physica D: Nonlinear Phenomena}{108}{1-2}{119-134}.
\PrintBackRefs{\CurrentBib}

\bibitem [\protect \citeauthoryear {%
Hornik%
, Stinchcombe%
\BCBL {}\ \BBA {} White%
}{%
Hornik%
\ \protect \BOthers {.}}{%
{\protect \APACyear {1989}}%
}]{%
hornik1989multilayer}
\APACinsertmetastar {%
hornik1989multilayer}%
\begin{APACrefauthors}%
Hornik, K.%
, Stinchcombe, M.%
\BCBL {}\ \BBA {} White, H.%
\end{APACrefauthors}%
\unskip\
\newblock
\APACrefYearMonthDay{1989}{}{}.
\newblock
{\BBOQ}\APACrefatitle {Multilayer feedforward networks are universal approximators} {Multilayer feedforward networks are universal approximators}.{\BBCQ}
\newblock
\APACjournalVolNumPages{Neural networks}{2}{5}{359--366}.
\PrintBackRefs{\CurrentBib}

\bibitem [\protect \citeauthoryear {%
Jin%
, Jeremiah%
, Su%
, Bai%
\BCBL {}\ \BBA {} Kong%
}{%
Jin%
\ \protect \BOthers {.}}{%
{\protect \APACyear {2021}}%
{\protect \APACexlab {{\protect \BCnt {1}}}}}]{%
jin2021new}
\APACinsertmetastar {%
jin2021new}%
\begin{APACrefauthors}%
Jin, X\BHBI B.%
, Jeremiah, R\BPBI J\BPBI R.%
, Su, T\BHBI L.%
, Bai, Y\BHBI T.%
\BCBL {}\ \BBA {} Kong, J\BHBI L.%
\end{APACrefauthors}%
\unskip\
\newblock
\APACrefYearMonthDay{2021{\protect \BCnt {1}}}{}{}.
\newblock
{\BBOQ}\APACrefatitle {The new trend of state estimation: from model-driven to hybrid-driven methods} {The new trend of state estimation: from model-driven to hybrid-driven methods}.{\BBCQ}
\newblock
\APACjournalVolNumPages{Sensors}{21}{6}{2085}.
\PrintBackRefs{\CurrentBib}

\bibitem [\protect \citeauthoryear {%
Jin%
, Jeremiah%
, Su%
, Bai%
\BCBL {}\ \BBA {} Kong%
}{%
Jin%
\ \protect \BOthers {.}}{%
{\protect \APACyear {2021}}%
{\protect \APACexlab {{\protect \BCnt {2}}}}}]{%
trends_of_estimation}
\APACinsertmetastar {%
trends_of_estimation}%
\begin{APACrefauthors}%
Jin, X\BHBI B.%
, Jeremiah, R\BPBI J\BPBI R.%
, Su, T\BHBI L.%
, Bai, Y\BHBI T.%
\BCBL {}\ \BBA {} Kong, J\BHBI L.%
\end{APACrefauthors}%
\unskip\
\newblock
\APACrefYearMonthDay{2021{\protect \BCnt {2}}}{}{}.
\newblock
{\BBOQ}\APACrefatitle {The new trend of state estimation: from model-driven to hybrid-driven methods} {The new trend of state estimation: from model-driven to hybrid-driven methods}.{\BBCQ}
\newblock
\APACjournalVolNumPages{Sensors}{21}{6}{2085}.
\PrintBackRefs{\CurrentBib}

\bibitem [\protect \citeauthoryear {%
Kalman%
}{%
Kalman%
}{%
{\protect \APACyear {1960}}%
}]{%
kalman1960new}
\APACinsertmetastar {%
kalman1960new}%
\begin{APACrefauthors}%
Kalman, R\BPBI E.%
\end{APACrefauthors}%
\unskip\
\newblock
\APACrefYearMonthDay{1960}{}{}.
\newblock
{\BBOQ}\APACrefatitle {A new approach to linear filtering and prediction problems} {A new approach to linear filtering and prediction problems}.{\BBCQ}
\newblock
\APACjournalVolNumPages{J. Basic Eng.}{82}{1}{35-45}.
\PrintBackRefs{\CurrentBib}

\bibitem [\protect \citeauthoryear {%
Kandiran%
\ \BBA {} Hacinliyan%
}{%
Kandiran%
\ \BBA {} Hacinliyan%
}{%
{\protect \APACyear {2019}}%
}]{%
kandiran2019comparison}
\APACinsertmetastar {%
kandiran2019comparison}%
\begin{APACrefauthors}%
Kandiran, E.%
\BCBT {}\ \BBA {} Hacinliyan, A.%
\end{APACrefauthors}%
\unskip\
\newblock
\APACrefYearMonthDay{2019}{}{}.
\newblock
{\BBOQ}\APACrefatitle {Comparison of feedforward and recurrent neural network in forecasting chaotic dynamical system} {Comparison of feedforward and recurrent neural network in forecasting chaotic dynamical system}.{\BBCQ}
\newblock
\APACjournalVolNumPages{AJIT-e: Academic Journal of Information Technology}{10}{37}{31--44}.
\PrintBackRefs{\CurrentBib}

\bibitem [\protect \citeauthoryear {%
Kang%
\ \BBA {} Wilcox%
}{%
Kang%
\ \BBA {} Wilcox%
}{%
{\protect \APACyear {2017}}%
}]{%
kang2017mitigating}
\APACinsertmetastar {%
kang2017mitigating}%
\begin{APACrefauthors}%
Kang, W.%
\BCBT {}\ \BBA {} Wilcox, L\BPBI C.%
\end{APACrefauthors}%
\unskip\
\newblock
\APACrefYearMonthDay{2017}{}{}.
\newblock
{\BBOQ}\APACrefatitle {Mitigating the curse of dimensionality: sparse grid characteristics method for optimal feedback control and {HJB} equations} {Mitigating the curse of dimensionality: sparse grid characteristics method for optimal feedback control and {HJB} equations}.{\BBCQ}
\newblock
\APACjournalVolNumPages{Computational Optimization and Applications}{68}{2}{289--315}.
\PrintBackRefs{\CurrentBib}

\bibitem [\protect \citeauthoryear {%
Kasiran%
, Ibrahim%
\BCBL {}\ \BBA {} Ribuan%
}{%
Kasiran%
\ \protect \BOthers {.}}{%
{\protect \APACyear {2012}}%
}]{%
kasiran2012mobile}
\APACinsertmetastar {%
kasiran2012mobile}%
\begin{APACrefauthors}%
Kasiran, Z.%
, Ibrahim, Z.%
\BCBL {}\ \BBA {} Ribuan, M\BPBI S\BPBI M.%
\end{APACrefauthors}%
\unskip\
\newblock
\APACrefYearMonthDay{2012}{}{}.
\newblock
{\BBOQ}\APACrefatitle {Mobile phone customers churn prediction using {E}lman and {J}ordan recurrent neural network} {Mobile phone customers churn prediction using {E}lman and {J}ordan recurrent neural network}.{\BBCQ}
\newblock
\APACjournalVolNumPages{7th international conference on computing and convergence technology (ICCCT)}{}{}{673-678}.
\PrintBackRefs{\CurrentBib}

\bibitem [\protect \citeauthoryear {%
Kaur%
, Zhou%
, Liu%
\BCBL {}\ \BBA {} Morris%
}{%
Kaur%
\ \protect \BOthers {.}}{%
{\protect \APACyear {2024}}%
}]{%
cdc_submission}
\APACinsertmetastar {%
cdc_submission}%
\begin{APACrefauthors}%
Kaur, A.%
, Zhou, R.%
, Liu, J.%
\BCBL {}\ \BBA {} Morris, K.%
\end{APACrefauthors}%
\unskip\
\newblock
\APACrefYearMonthDay{2024}{}{}.
\newblock
{\BBOQ}\APACrefatitle {Stability of {J}ordan recurrent neural network estimator} {Stability of {J}ordan recurrent neural network estimator}.{\BBCQ}
\newblock
\APACjournalVolNumPages{preprint}{}{}{}.
\PrintBackRefs{\CurrentBib}

\bibitem [\protect \citeauthoryear {%
Krener%
}{%
Krener%
}{%
{\protect \APACyear {2003}}%
{\protect \APACexlab {{\protect \BCnt {1}}}}}]{%
krener_2002}
\APACinsertmetastar {%
krener_2002}%
\begin{APACrefauthors}%
Krener, A\BPBI J.%
\end{APACrefauthors}%
\unskip\
\newblock
\APACrefYearMonthDay{2003{\protect \BCnt {1}}}{}{}.
\newblock
{\BBOQ}\APACrefatitle {The convergence of the extended {K}alman filter} {The convergence of the extended {K}alman filter}.{\BBCQ}
\newblock
\APACjournalVolNumPages{Directions in Mathematical Systems Theory and Optimization (Lecture Notes in Control and Information Sciences) by Rantzer and Byrnes, Eds. Berlin, Germany}{286}{}{173-182}.
\PrintBackRefs{\CurrentBib}

\bibitem [\protect \citeauthoryear {%
Krener%
}{%
Krener%
}{%
{\protect \APACyear {2003}}%
{\protect \APACexlab {{\protect \BCnt {2}}}}}]{%
krener_2003}
\APACinsertmetastar {%
krener_2003}%
\begin{APACrefauthors}%
Krener, A\BPBI J.%
\end{APACrefauthors}%
\unskip\
\newblock
\APACrefYearMonthDay{2003{\protect \BCnt {2}}}{}{}.
\newblock
{\BBOQ}\APACrefatitle {The convergence of the minimum energy estimator} {The convergence of the minimum energy estimator}.{\BBCQ}
\newblock
\APACjournalVolNumPages{New Trends in Nonlinear Dynamics and Control and their Applications}{}{}{187-208}.
\PrintBackRefs{\CurrentBib}

\bibitem [\protect \citeauthoryear {%
Kuan%
, Hornik%
\BCBL {}\ \BBA {} White%
}{%
Kuan%
\ \protect \BOthers {.}}{%
{\protect \APACyear {1994}}%
}]{%
kuan_convergence_1994}
\APACinsertmetastar {%
kuan_convergence_1994}%
\begin{APACrefauthors}%
Kuan, C\BHBI M.%
, Hornik, K.%
\BCBL {}\ \BBA {} White, H.%
\end{APACrefauthors}%
\unskip\
\newblock
\APACrefYearMonthDay{1994}{}{}.
\newblock
{\BBOQ}\APACrefatitle {A convergence result for learning in recurrent neural networks} {A convergence result for learning in recurrent neural networks}.{\BBCQ}
\newblock
\APACjournalVolNumPages{Neural Computation}{6}{3}{420--440}.
\newblock
\begin{APACrefDOI} \doi{10.1162/neco.1994.6.3.420} \end{APACrefDOI}
\PrintBackRefs{\CurrentBib}

\bibitem [\protect \citeauthoryear {%
Kutschireiter%
, Surace%
\BCBL {}\ \BBA {} Pfister%
}{%
Kutschireiter%
\ \protect \BOthers {.}}{%
{\protect \APACyear {2020}}%
}]{%
kutschireiter2020hitchhiker}
\APACinsertmetastar {%
kutschireiter2020hitchhiker}%
\begin{APACrefauthors}%
Kutschireiter, A.%
, Surace, S\BPBI C.%
\BCBL {}\ \BBA {} Pfister, J\BHBI P.%
\end{APACrefauthors}%
\unskip\
\newblock
\APACrefYearMonthDay{2020}{}{}.
\newblock
{\BBOQ}\APACrefatitle {The {H}itchhiker’s guide to nonlinear filtering} {The {H}itchhiker’s guide to nonlinear filtering}.{\BBCQ}
\newblock
\APACjournalVolNumPages{Journal of Mathematical Psychology}{94}{}{102307}.
\PrintBackRefs{\CurrentBib}

\bibitem [\protect \citeauthoryear {%
Manaswi%
}{%
Manaswi%
}{%
{\protect \APACyear {2018}}%
}]{%
manaswi2018rnn}
\APACinsertmetastar {%
manaswi2018rnn}%
\begin{APACrefauthors}%
Manaswi, N\BPBI K.%
\end{APACrefauthors}%
\unskip\
\newblock
\APACrefYearMonthDay{2018}{}{}.
\newblock
{\BBOQ}\APACrefatitle {{RNN} and {LSTM})} {{RNN} and {LSTM})}.{\BBCQ}
\newblock
\APACjournalVolNumPages{Deep Learning with Applications Using Python: Chatbots and Face, Object, and Speech Recognition With TensorFlow and Keras}{}{}{115--126}.
\PrintBackRefs{\CurrentBib}

\bibitem [\protect \citeauthoryear {%
Moireau%
}{%
Moireau%
}{%
{\protect \APACyear {2018}}%
}]{%
moireau}
\APACinsertmetastar {%
moireau}%
\begin{APACrefauthors}%
Moireau, P.%
\end{APACrefauthors}%
\unskip\
\newblock
\APACrefYearMonthDay{2018}{}{}.
\newblock
{\BBOQ}\APACrefatitle {A discrete-time optimal filtering approach for non-linear systems as a stable discretization of the {M}ortensen observer} {A discrete-time optimal filtering approach for non-linear systems as a stable discretization of the {M}ortensen observer}.{\BBCQ}
\newblock
\APACjournalVolNumPages{ESAIM: Control, Optimisation and Calculus of Variations}{24}{4}{1815-1847}.
\PrintBackRefs{\CurrentBib}

\bibitem [\protect \citeauthoryear {%
Mortensen%
}{%
Mortensen%
}{%
{\protect \APACyear {1968}}%
}]{%
mortensen}
\APACinsertmetastar {%
mortensen}%
\begin{APACrefauthors}%
Mortensen, R.%
\end{APACrefauthors}%
\unskip\
\newblock
\APACrefYearMonthDay{1968}{}{}.
\newblock
{\BBOQ}\APACrefatitle {Maximum-likelihood recursive nonlinear filtering} {Maximum-likelihood recursive nonlinear filtering}.{\BBCQ}
\newblock
\APACjournalVolNumPages{Journal of Optimization Theory and Applications}{2}{6}{386-394}.
\PrintBackRefs{\CurrentBib}

\bibitem [\protect \citeauthoryear {%
Nair%
\ \BBA {} Goza%
}{%
Nair%
\ \BBA {} Goza%
}{%
{\protect \APACyear {2020}}%
}]{%
nair_goza}
\APACinsertmetastar {%
nair_goza}%
\begin{APACrefauthors}%
Nair, N\BPBI J.%
\BCBT {}\ \BBA {} Goza, A.%
\end{APACrefauthors}%
\unskip\
\newblock
\APACrefYearMonthDay{2020}{}{}.
\newblock
{\BBOQ}\APACrefatitle {Leveraging reduced-order models for state estimation using deep learning} {Leveraging reduced-order models for state estimation using deep learning}.{\BBCQ}
\newblock
\APACjournalVolNumPages{Journal of Fluid Mechanics}{897}{}{}.
\PrintBackRefs{\CurrentBib}

\bibitem [\protect \citeauthoryear {%
Nakamura-Zimmerer%
, Gong%
\BCBL {}\ \BBA {} Kang%
}{%
Nakamura-Zimmerer%
\ \protect \BOthers {.}}{%
{\protect \APACyear {2021}}%
}]{%
nakamura2021adaptive}
\APACinsertmetastar {%
nakamura2021adaptive}%
\begin{APACrefauthors}%
Nakamura-Zimmerer, T.%
, Gong, Q.%
\BCBL {}\ \BBA {} Kang, W.%
\end{APACrefauthors}%
\unskip\
\newblock
\APACrefYearMonthDay{2021}{}{}.
\newblock
{\BBOQ}\APACrefatitle {Adaptive deep learning for high-dimensional {H}amilton-{J}acobi-{B}ellman equations} {Adaptive deep learning for high-dimensional {H}amilton-{J}acobi-{B}ellman equations}.{\BBCQ}
\newblock
\APACjournalVolNumPages{SIAM Journal on Scientific Computing}{43}{2}{A1221-A1247}.
\PrintBackRefs{\CurrentBib}

\bibitem [\protect \citeauthoryear {%
Park%
, Yi%
\BCBL {}\ \BBA {} Ji%
}{%
Park%
\ \protect \BOthers {.}}{%
{\protect \APACyear {2020}}%
}]{%
park2020analysis}
\APACinsertmetastar {%
park2020analysis}%
\begin{APACrefauthors}%
Park, J.%
, Yi, D.%
\BCBL {}\ \BBA {} Ji, S.%
\end{APACrefauthors}%
\unskip\
\newblock
\APACrefYearMonthDay{2020}{}{}.
\newblock
{\BBOQ}\APACrefatitle {Analysis of recurrent neural network and predictions} {Analysis of recurrent neural network and predictions}.{\BBCQ}
\newblock
\APACjournalVolNumPages{Symmetry}{12}{4}{615}.
\PrintBackRefs{\CurrentBib}

\bibitem [\protect \citeauthoryear {%
Pequito%
, Aguiar%
, Sinopoli%
\BCBL {}\ \BBA {} Gomes%
}{%
Pequito%
\ \protect \BOthers {.}}{%
{\protect \APACyear {2011}}%
}]{%
pequito}
\APACinsertmetastar {%
pequito}%
\begin{APACrefauthors}%
Pequito, S.%
, Aguiar, A\BPBI P.%
, Sinopoli, B.%
\BCBL {}\ \BBA {} Gomes, D\BPBI A.%
\end{APACrefauthors}%
\unskip\
\newblock
\APACrefYearMonthDay{2011}{}{}.
\newblock
{\BBOQ}\APACrefatitle {Nonlinear estimation using mean field games} {Nonlinear estimation using mean field games}.{\BBCQ}
\newblock
\APACjournalVolNumPages{International Conference on NETwork Games, Control and Optimization (NetGCooP 2011)}{}{}{1-5}.
\PrintBackRefs{\CurrentBib}

\bibitem [\protect \citeauthoryear {%
Peralez%
\ \BBA {} Nadri%
}{%
Peralez%
\ \BBA {} Nadri%
}{%
{\protect \APACyear {2021}}%
}]{%
dl_based_observer}
\APACinsertmetastar {%
dl_based_observer}%
\begin{APACrefauthors}%
Peralez, J.%
\BCBT {}\ \BBA {} Nadri, M.%
\end{APACrefauthors}%
\unskip\
\newblock
\APACrefYearMonthDay{2021}{}{}.
\newblock
{\BBOQ}\APACrefatitle {Deep learning-based {L}uenberger observer design for discrete-time nonlinear systems} {Deep learning-based {L}uenberger observer design for discrete-time nonlinear systems}.{\BBCQ}
\newblock
\APACjournalVolNumPages{60th IEEE Conference on Design and Control}{}{}{4370-4375}.
\PrintBackRefs{\CurrentBib}

\bibitem [\protect \citeauthoryear {%
Sch{\"a}fer%
\ \BBA {} Zimmermann%
}{%
Sch{\"a}fer%
\ \BBA {} Zimmermann%
}{%
{\protect \APACyear {2007}}%
}]{%
schafer2007recurrent}
\APACinsertmetastar {%
schafer2007recurrent}%
\begin{APACrefauthors}%
Sch{\"a}fer, A\BPBI M.%
\BCBT {}\ \BBA {} Zimmermann, H\BHBI G.%
\end{APACrefauthors}%
\unskip\
\newblock
\APACrefYearMonthDay{2007}{}{}.
\newblock
{\BBOQ}\APACrefatitle {Recurrent neural networks are universal approximators} {Recurrent neural networks are universal approximators}.{\BBCQ}
\newblock
\APACjournalVolNumPages{International journal of neural systems}{17}{04}{253--263}.
\PrintBackRefs{\CurrentBib}

\bibitem [\protect \citeauthoryear {%
Simon%
}{%
Simon%
}{%
{\protect \APACyear {2006}}%
}]{%
dansimon}
\APACinsertmetastar {%
dansimon}%
\begin{APACrefauthors}%
Simon, D.%
\end{APACrefauthors}%
\unskip\
\newblock
\APACrefYear{2006}.
\newblock
\APACrefbtitle {Optimal state estimation: {K}alman, ${H}^{\infty}$, and nonlinear approaches} {Optimal state estimation: {K}alman, ${H}^{\infty}$, and nonlinear approaches}.
\newblock
\APACaddressPublisher{}{Wiley-Interscience}.
\PrintBackRefs{\CurrentBib}

\bibitem [\protect \citeauthoryear {%
Slade%
, Sunberg%
\BCBL {}\ \BBA {} Kochenderfer%
}{%
Slade%
\ \protect \BOthers {.}}{%
{\protect \APACyear {2020}}%
}]{%
patrick}
\APACinsertmetastar {%
patrick}%
\begin{APACrefauthors}%
Slade, P.%
, Sunberg, Z\BPBI N.%
\BCBL {}\ \BBA {} Kochenderfer, M\BPBI J.%
\end{APACrefauthors}%
\unskip\
\newblock
\APACrefYearMonthDay{2020}{}{}.
\newblock
{\BBOQ}\APACrefatitle {Estimation and control using sampling-based {B}ayesian reinforcement learning} {Estimation and control using sampling-based {B}ayesian reinforcement learning}.{\BBCQ}
\newblock
\APACjournalVolNumPages{IET Cyber-Physical Systems: Theory \& Applications}{}{}{}.
\PrintBackRefs{\CurrentBib}

\bibitem [\protect \citeauthoryear {%
Wu%
, An%
, Guan%
, Huang%
\BCBL {}\ \BBA {} Zhou%
}{%
Wu%
\ \protect \BOthers {.}}{%
{\protect \APACyear {2019}}%
}]{%
wu2019time}
\APACinsertmetastar {%
wu2019time}%
\begin{APACrefauthors}%
Wu, W.%
, An, S\BHBI Y.%
, Guan, P.%
, Huang, D\BHBI S.%
\BCBL {}\ \BBA {} Zhou, B\BHBI S.%
\end{APACrefauthors}%
\unskip\
\newblock
\APACrefYearMonthDay{2019}{}{}.
\newblock
{\BBOQ}\APACrefatitle {Time series analysis of human brucellosis in mainland {C}hina by using {E}lman and {J}ordan recurrent neural networks} {Time series analysis of human brucellosis in mainland {C}hina by using {E}lman and {J}ordan recurrent neural networks}.{\BBCQ}
\newblock
\APACjournalVolNumPages{BMC infectious diseases}{19}{}{1--11}.
\PrintBackRefs{\CurrentBib}

\bibitem [\protect \citeauthoryear {%
Xie%
\ \BBA {} Zhang%
}{%
Xie%
\ \BBA {} Zhang%
}{%
{\protect \APACyear {2021}}%
}]{%
xie_deep_2021}
\APACinsertmetastar {%
xie_deep_2021}%
\begin{APACrefauthors}%
Xie, B.%
\BCBT {}\ \BBA {} Zhang, Q.%
\end{APACrefauthors}%
\unskip\
\newblock
\APACrefYearMonthDay{2021}{}{}.
\newblock
{\BBOQ}\APACrefatitle {Deep filtering with {DNN}, {CNN} and {RNN}} {Deep filtering with {DNN}, {CNN} and {RNN}}.{\BBCQ}
\newblock
\APACjournalVolNumPages{arXiv}{}{}{}.
\PrintBackRefs{\CurrentBib}

\bibitem [\protect \citeauthoryear {%
Yadaiah%
, Bapi%
, Singh%
\BCBL {}\ \BBA {} Deekshatulu%
}{%
Yadaiah%
\ \protect \BOthers {.}}{%
{\protect \APACyear {2011}}%
}]{%
yadaiah}
\APACinsertmetastar {%
yadaiah}%
\begin{APACrefauthors}%
Yadaiah, N.%
, Bapi, R\BPBI S.%
, Singh, L.%
\BCBL {}\ \BBA {} Deekshatulu, B\BPBI L.%
\end{APACrefauthors}%
\unskip\
\newblock
\APACrefYearMonthDay{2011}{}{}.
\newblock
{\BBOQ}\APACrefatitle {{DEKF} based recurrent neural network for state estimation of nonlinear dynamical systems} {{DEKF} based recurrent neural network for state estimation of nonlinear dynamical systems}.{\BBCQ}
\newblock
\APACjournalVolNumPages{IEEE Recent Advances in Intelligent Computational Systems}{}{}{311-316}.
\newblock
\begin{APACrefDOI} \doi{10.1109/RAICS.2011.6069325} \end{APACrefDOI}
\PrintBackRefs{\CurrentBib}

\end{thebibliography}

\end{document}